\documentclass[final,3p]{elsarticle}
 \usepackage{graphics}
 \usepackage{graphicx}
 \usepackage{epsfig}
\usepackage{amssymb}
 \usepackage{amsthm}
 \usepackage{lineno}
 \usepackage{amsmath}
   \numberwithin{equation}{section}
\usepackage{mathrsfs}

\newtheorem{thm}{Theorem}[section]

\newtheorem{lem}[thm]{Lemma}

\newtheorem{defn}[thm]{Definition}

\biboptions{sort&compress,square}
\allowdisplaybreaks
\begin{document}
\begin{frontmatter}
\author{Tong Wu$^1$}
\ead{wut977@nenu.edu.cn}
\author{Yong Wang\corref{cor2}}
\ead{wangy581@nenu.edu.cn}
\cortext[cor2]{Corresponding author.}

\address{School of Mathematics and Statistics, Northeast Normal University,
Changchun, 130024, China}
\title{The operator $\sqrt{-1}\widehat{c}(V)(d+\delta)$ and the Kastler-Kalau-Walze type theorems}
\begin{abstract}
In this paper, we obtain two Lichnerowicz type formulas for the operators $\sqrt{-1}\widehat{c}(V)(d+\delta)$ and $-\sqrt{-1}(d+\delta)\widehat{c}(V)$. And we give the proof of Kastler-Kalau-Walze type theorems for the operators $\sqrt{-1}\widehat{c}(V)(d+\delta)$ and $-\sqrt{-1}(d+\delta)\widehat{c}(V)$ on 3,4-dimensional oriented compact manifolds with (resp.without) boundary.
\end{abstract}
\begin{keyword} Lichnerowicz type formulas; the operators $\sqrt{-1}\widehat{c}(V)(d+\delta)$ and $-\sqrt{-1}(d+\delta)\widehat{c}(V)$; Kastler-Kalau-Walze type theorems.\\

\end{keyword}
\end{frontmatter}
\section{Introduction}
 Until now, many geometers have studied noncommutative residues. In \cite{Gu,Wo}, authors found noncommutative residues are of great importance to the study of noncommutative geometry. In \cite{Co1}, Connes used the noncommutative residue to derive a conformal 4-dimensional Polyakov action analogy. Connes showed us that the noncommutative residue on a compact manifold $M$ coincided with the Dixmier's trace on pseudodifferential operators of order $-{\rm {dim}}M$ in \cite{Co2}.
And Connes claimed the noncommutative residue of the square of the inverse of the Dirac operator was proportioned to the Einstein-Hilbert action.  Kastler \cite{Ka} gave a
brute-force proof of this theorem. Kalau and Walze proved this theorem in the normal coordinates system simultaneously in \cite{KW} .
Ackermann proved that
the Wodzicki residue  of the square of the inverse of the Dirac operator ${\rm  Wres}(D^{-2})$ in turn is essentially the second coefficient
of the heat kernel expansion of $D^{2}$ in \cite{Ac}.

On the other hand, Wang generalized the Connes' results to the case of manifolds with boundary in \cite{Wa1,Wa2},
and proved the Kastler-Kalau-Walze type theorem for the Dirac operator and the signature operator on lower-dimensional manifolds
with boundary \cite{Wa3}. In \cite{Wa3,Wa4}, Wang computed $\widetilde{{\rm Wres}}[\pi^+D^{-1}\circ\pi^+D^{-1}]$ and $\widetilde{{\rm Wres}}[\pi^+D^{-2}\circ\pi^+D^{-2}]$, where the two operators are symmetric, in these cases the boundary term vanished. But for $\widetilde{{\rm Wres}}[\pi^+D^{-1}\circ\pi^+D^{-3}]$, Wang got a nonvanishing boundary term \cite{Wa5}, and give a theoretical explanation for gravitational action on boundary. In others words, Wang provides a kind of method to study the Kastler-Kalau-Walze type theorem for manifolds with boundary.
In \cite{lkl}, L\'{o}pez and his collaborators introduced an elliptic differential operator which is called the Novikov operator. In \cite{WW}, Wei and Wang proved Kastler-Kalau-Walze type theorem for modified Novikov operators on compact manifolds. In \cite{Wa3}, the leading symbol of the Dirac operator $D$ is $\sqrt{-1}c(\xi)$. To get the leading symbol of the operator which is not $\sqrt{-1}c(\xi)$, we consider the operator $\sqrt{-1}\widehat{c}(V)(d+\delta)$ in this paper, which is motivated by the sub-signature operator in \cite{wpz1}. \\
\indent In this paper, we obtain two Lichnerowicz type formulas for the operators $\sqrt{-1}\widehat{c}(V)(d+\delta)$ and $-\sqrt{-1}(d+\delta)\widehat{c}(V)$, and prove the following main theorems.\\
\begin{thm}
\label{thm1.1}
Let $M$ be a $4$-dimensional oriented
compact manifold with boundary $\partial M$ and the metric
$g^{TM}$ as in Section \ref{Section:3}, the operators ${D_V}=\sqrt{-1}\widehat{c}(V)(d +\delta)$ and ${D_V}^*=-\sqrt{-1}(d +\delta)\widehat{c}(V)$ be on $\widetilde{M}$ ($\widetilde{M}$ is a collar neighborhood of $M$), then
\begin{align}
\label{ali1.1}
&\widetilde{{\rm Wres}}[\pi^+{D_V}^{-1}\circ\pi^+({D_V}^*)^{-1}]\nonumber\\
&=32\pi^2\int_{M}\bigg(-\frac{4}{3}K\bigg)d{\rm Vol_{M}}+\int_{\partial M}\left(-\frac{3ih'(0)}{2}-\frac{27\pi^2h'(0)}{8}-\frac{\pi^2}{4}\right)\pi\Omega_3d{\rm Vol_{M}},
\end{align}
where $K$ is the scalar curvature, $h$ is defined by (\ref{a1}) and ${\rm \Omega_{3}}$ is the canonical volume of $S^{3}$.
\end{thm}
\begin{thm}
\label{thm1.2}
Let $M$ be a $4$-dimensional oriented
compact manifold with boundary $\partial M$ and the metric
$g^{TM}$ as in Section \ref{Section:3}, the operator ${D_V}=\sqrt{-1}\widehat{c}(V)(d +\delta)$ be on $\widetilde{M}$ ($\widetilde{M}$ is a collar neighborhood of $M$), then
\begin{align}
\label{ali1.2}
&\widetilde{{\rm Wres}}[\pi^+{D_V}^{-1}\circ\pi^+{D_V}^{-1}]\nonumber\\
&=32\pi^2\int_{M}\bigg(
-\frac{4}{3}K-8\sum_{q=1}^4|\nabla^L_{e_q}V|^2\bigg)d{\rm Vol_{M}}+\int_{\partial M}\left(-\frac{3ih'(0)}{2}-\frac{27\pi^2h'(0)}{8}-\frac{\pi^2}{4}\right)\pi\Omega_3d{\rm Vol_{M}},\nonumber\\
\end{align}
where $K$ is the scalar curvature, $h$ is defined by (\ref{a1}) and ${\rm \Omega_{3}}$ is the canonical volume of $S^{3}$.
\end{thm}
\indent We note that two operators in Theorem \ref{thm1.2} are symmetric, but we still get the non-vanishing boundary term.\\
\begin{thm}\label{thm1.3}
Let $M$ be a $3$-dimensional
oriented compact manifold with boundary $\partial M$ and the metric
$g^{TM}$ as in Section \ref{Section:3}, the operator ${D_V}=\sqrt{-1}\widehat{c}(V)(d +\delta)$ be on $\widetilde{M}$ ($\widetilde{M}$ is a collar neighborhood of $M$), then
\begin{align}\label{b7}
\widetilde{{\rm Wres}}[(\pi^+D_V^{-1})^2]=4i\pi^2vol_{\partial M},\nonumber\\
\end{align}
where $vol_{\partial M}$ denotes the canonical volume form of $\partial M$.\\
\end{thm}
\indent The paper is organized in the following way. In Section \ref{Section:2}, by using the definition of the operators $\sqrt{-1}\widehat{c}(V)(d+\delta)$ and $-\sqrt{-1}(d+\delta)\widehat{c}(V)$, we compute the Lichnerowicz formulas for the operators $\sqrt{-1}\widehat{c}(V)(d+\delta)$ and $-\sqrt{-1}(d+\delta)\widehat{c}(V)$. In Section \ref{Section:3},
 we prove the Kastler-Kalau-Walze type theorem for 4-dimensional manifolds with boundary for the operators $\sqrt{-1}\widehat{c}(V)(d+\delta)$ and $-\sqrt{-1}(d+\delta)\widehat{c}(V)$. In Section \ref{Section:4}, we compute $\widetilde{Wres}[(\pi^+(D_V)^{-1})^2]$ for 3-dimensional oriented manifolds with boundary.
\section{The operator $\sqrt{-1}\widehat{c}(V)(d+\delta)$ and its Lichnerowicz formula}
\label{Section:2}
Firstly we introduce some notations about the operators $\sqrt{-1}\widehat{c}(V)(d+\delta)$ and $-\sqrt{-1}(d+\delta)\widehat{c}(V)$. Let $M$ be an $n$-dimensional ($n\geq 3$) oriented compact Riemannian manifold with a Riemannian metric $g^{TM}$.\\
\indent Let $\nabla^L$ be the Levi-Civita connection about $g^{TM}$. In the
fixed orthonormal frame $\{e_1,\cdots,e_n\}$, the connection matrix $(\omega_{s,t})$ is defined by
\begin{equation}
\label{eq1}
\nabla^L(e_1,\cdots,e_n)= (e_1,\cdots,e_n)(\omega_{s,t}).\\
\end{equation}
\indent Let $\epsilon (e_j^*)$,~$\iota (e_j^*)$ be the exterior and interior multiplications respectively, where $e_j^*=g^{TM}(e_j,\cdot)$. And $c(e_j)$ be the Clifford action,\\
write
\begin{equation}
\label{eq2}
\widehat{c}(e_j)=\epsilon (e_j^* )+\iota
(e_j^*);~~
c(e_j)=\epsilon (e_j^* )-\iota (e_j^* ),
\end{equation}
which satisfies
\begin{align}
\label{ali1}
&\widehat{c}(e_i)\widehat{c}(e_j)+\widehat{c}(e_j)\widehat{c}(e_i)=2g^{TM}(e_i,e_j);~~\nonumber\\
&c(e_i)c(e_j)+c(e_j)c(e_i)=-2g^{TM}(e_i,e_j);~~\nonumber\\
&c(e_i)\widehat{c}(e_j)+\widehat{c}(e_j)c(e_i)=0.\nonumber\\
\end{align}
By \cite{Y}, we have
\begin{align}
\label{ali2}
D&=d+\delta=\sum^n_{i=1}c(e_i)\bigg[e_i+\frac{1}{4}\sum_{s,t}\omega_{s,t}
(e_i)[\widehat{c}(e_s)\widehat{c}(e_t)
-c(e_s)c(e_t)]\bigg].
\end{align}
\indent Let $e_1,e_2\cdot\cdot\cdot,e_n$ be the orthonormal basis of $TM$, the operators ${D}_V$ and ${D_V}^*$ acting on $\bigwedge^*T^*M\bigotimes\mathbb{C}$ are defined by
\begin{align}
\label{ali3}
D_V&=\sqrt{-1}\widehat{c}(V)(d +\delta)\nonumber\\
&=\sqrt{-1}\widehat{c}(V)\sum^n_{i=1}c(e_i)\bigg[e_i+\frac{1}{4}\sum_{s,t}\omega_{s,t}
(e_i)[\widehat{c}(e_s)\widehat{c}(e_t)
-c(e_s)c(e_t)]\bigg],\nonumber\\
\end{align}
\begin{align}
\label{ali4}
{D_V}^*&=-\sqrt{-1}(d+\delta)\widehat{c}(V)\nonumber\\
&=-\sqrt{-1}\sum^n_{i=1}c(e_i)\bigg[e_i+\frac{1}{4}\sum_{s,t}\omega_{s,t}
(e_i)[\widehat{c}(e_s)\widehat{c}(e_t)
-c(e_s)c(e_t)]\bigg]\widehat{c}(V),\nonumber\\
\end{align}
where $V$ is a vector field, and $|V|=1$.\\
\indent Next, we get the following theorem about Lichnerowicz formulas,
\begin{thm}
\label{thm1} The following equalities hold:
\begin{align}
\label{ali5}
{D_V}^*D_V&=-\Big[g^{ij}(\nabla_{\partial_{i}}\nabla_{\partial_{j}}-
\nabla_{\nabla^{L}_{\partial_{i}}\partial_{j}})\Big]
-\frac{1}{8}\sum_{ijkl}R_{ijkl}\widehat{c}(e_i)\widehat{c}(e_j)
c(e_k)c(e_l)+\frac{1}{4}K;\nonumber\\
{D_V}^2&=-\Big[g^{ij}(\nabla_{\partial_{i}}\nabla_{\partial_{j}}-
\nabla_{\nabla^{L}_{\partial_{i}}\partial_{j}})\Big]
-\frac{1}{8}\sum_{ijkl}R_{ijkl}\widehat{c}(e_i)\widehat{c}(e_j)
c(e_k)c(e_l)+\frac{1}{4}K\nonumber\\
&+\frac{1}{4}\sum_{i=1}^n[\widehat{c}(V)\sum_{q=1}^n\widehat{c}(\nabla^L_{e_q}V)c(e_q)c(e_i)]^2-\frac{1}{2}\sum_{j=1}^n[\nabla^{\bigwedge^*T^*M}_{e_j}(\widehat{c}(V)\sum_{q=1}^n\widehat{c}(\nabla^L_{e_q}V)c(e_q))c(e_j)],\nonumber\\
\end{align}
where $K$ is the scalar curvature.
\end{thm}
\begin{proof}
Let $M$ be a smooth compact oriented Riemannian $n$-dimensional manifolds without boundary and $N$ be a vector bundle on $M$. If $P$ is a differential operator of Laplace type, then it has locally the form
\begin{equation}\label{eq3}
P=-(g^{ij}\partial_i\partial_j+A^i\partial_i+B),
\end{equation}
where $\partial_{i}$  is a natural local frame on $TM$
and $(g^{ij})_{1\leq i,j\leq n}$ is the inverse matrix associated to the metric
matrix  $(g_{ij})_{1\leq i,j\leq n}$ on $M$,
 and $A^{i}$ and $B$ are smooth sections
of $\textrm{End}(N)$ on $M$ (endomorphism). If a Laplace type
operator $P$  satisfies (\ref{eq3}), then there is a unique
connection $\nabla$ on $N$ and a unique endomorphism $E$ such that
 \begin{equation}
 \label{eq4}
P=-\Big[g^{ij}(\nabla_{\partial_{i}}\nabla_{\partial_{j}}-
 \nabla_{\nabla^{L}_{\partial_{i}}\partial_{j}})+E\Big],
\end{equation}
where $\nabla^{L}$ is the Levi-Civita connection on $M$. Moreover
(with local frames of $T^{*}M$ and $N$), $\nabla_{\partial_{i}}=\partial_{i}+\omega_{i} $
and $E$ are related to $g^{ij}$, $A^{i}$ and $B$ through
 \begin{eqnarray}
 \label{eq5}
&&\omega_{i}=\frac{1}{2}g_{ij}\big(A^{i}+g^{kl}\Gamma_{ kl}^{j} \texttt{id}\big),\\
&&E=B-g^{ij}\big(\partial_{i}(\omega_{j})+\omega_{i}\omega_{j}-\omega_{k}\Gamma_{ ij}^{k} \big),
\end{eqnarray}
where $\Gamma_{ kl}^{j}$ is the  Christoffel coefficient of $\nabla^{L}$.\\
\indent Let $g^{ij}=g(dx_{i},dx_{j})$, $\xi=\sum_{j}\xi_{j}dx_{j}$ and $\nabla^L_{\partial_{i}}\partial_{j}=\sum_{k}\Gamma_{ij}^{k}\partial_{k}$,  we denote that
\begin{align}
\label{ali6}
&\sigma_{i}=-\frac{1}{4}\sum_{s,t}\omega_{s,t}
(e_i)c(e_s)c(e_t)
;~~~a_{i}=\frac{1}{4}\sum_{s,t}\omega_{s,t}
(e_i)\widehat{c}(e_s)\widehat{c}(e_t);\nonumber\\
&\xi^{j}=g^{ij}\xi_{i};~~~~\Gamma^{k}=g^{ij}\Gamma_{ij}^{k};~~~~\sigma^{j}=g^{ij}\sigma_{i};
~~~~a^{j}=g^{ij}a_{i}.
\end{align}
\indent Then the operators $D_V$ and ${D_V}^*$ can be written as
\begin{equation}
\label{eq6}
D_V=\sqrt{-1}\widehat{c}(V)\sum^n_{i=1}c(e_i)[e_i+a_{i}+\sigma_{i}];
\end{equation}
\begin{equation}\label{eq7}
{D_V}^*=-\sqrt{-1}\sum^n_{i=1}c(e_i)[e_i+a_{i}+\sigma_{i}]\widehat{c}(V).
\end{equation}
\indent By \cite{Ac} and  \cite{Y}, we have
\begin{equation}\label{eq8}
(d+\delta)^{2}
=-\triangle_{0}-\frac{1}{8}\sum_{ijkl}R_{ijkl}\widehat{c}(e_i)\widehat{c}(e_j)c(e_k)c(e_l)+\frac{1}{4}K;
\end{equation}
\begin{equation}
\label{eq9}
-\triangle_{0}=\Delta=-g^{ij}(\nabla^L_{i}\nabla^L_{j}-\Gamma_{ij}^{k}\nabla^L_{k}).
\end{equation}
\indent By (\ref{eq3}) and (\ref{eq4}), we have
\begin{align}\label{eq10}
{D_V}^*D_V&=(d+\delta)^2\nonumber\\
&=-\sum_{ij}g^{ij}\Big[\partial_{i}\partial_{j}+2\sigma_{i}\partial_{j}+2a_{i}\partial_{j}-\Gamma_{ij}^{k}\partial_{k}+(\partial_{i}\sigma_{j})
+(\partial_{i}a_{j})
+\sigma_{i}\sigma_{j}+\sigma_{i}a_{j}+a_{i}\sigma_{j}+a_{i}a_{j} -\Gamma_{ij}^{k}\sigma_{k}\nonumber\\
&-\Gamma_{ij}^{k}a_{k}\Big]
-\frac{1}{8}\sum_{ijkl}R_{ijkl}\widehat{c}(e_i)\widehat{c}(e_j)
c(e_k)c(e_l)+\frac{1}{4}K.\nonumber\\
\end{align}
\indent Similarly, by $d+\delta=\sum_{q=1}^nc(e_q)\nabla^{\bigwedge^*T^*M}_{e_q}$, we have
\begin{align}\label{eq11}
{D_V}^2&=(d+\delta)^2+\sum_{q=1}^n\widehat{c}(V)\widehat{c}(\nabla^L_{e_q}V)c(e_q)(d+\delta)\nonumber\\
&=-\sum_{ij}g^{ij}\Big[\partial_{i}\partial_{j}+2\sigma_{i}\partial_{j}+2a_{i}\partial_{j}-\Gamma_{ij}^{k}\partial_{k}+(\partial_{i}\sigma_{j})
+(\partial_{i}a_{j})
+\sigma_{i}\sigma_{j}+\sigma_{i}a_{j}+a_{i}\sigma_{j}+a_{i}a_{j} -\Gamma_{ij}^{k}\sigma_{k}\nonumber\\
&-\Gamma_{ij}^{k}a_{k}\Big]
+\sum_{ij}g^{ij}\Big[\widehat{c}(V)\sum_{q=1}^n\widehat{c}(\nabla^L_{e_q}V)c(e_q)c(\partial_i)\partial_j
+\widehat{c}(V)\sum_{q=1}^n\widehat{c}(\nabla^L_{e_q}V)c(e_q)c(\partial_i)\sigma_i+\widehat{c}(V)\sum_{q=1}^n\widehat{c}(\nabla^L_{e_q}V)c(e_q)\nonumber\\
&c(\partial_i)a_i\Big]-\frac{1}{8}\sum_{ijkl}R_{ijkl}\widehat{c}(e_i)\widehat{c}(e_j)
c(e_k)c(e_l)+\frac{1}{4}K.\nonumber\\
\end{align}
\indent By (\ref{ali6})-(\ref{eq11}), we have
\begin{align}\label{eq12}
(\omega_{i})_{{D_V}^*{D_V}}&=\sigma_{i}+a_{i},\nonumber\\
\end{align}
\begin{align}
E_{{D_V}^*D_V}&=\frac{1}{8}\sum_{ijkl}R_{ijkl}\widehat{c}(e_i)\widehat{c}(e_j)c(e_k)c(e_l)-\frac{1}{4}K.\nonumber\\
\end{align}
\indent Since $E$ is globally
defined on $M$, taking normal coordinates at $x_0$, we have
$\sigma^{i}(x_0)=0$, $a^{i}(x_0)=0$, $\partial^{j}[c(\partial_{j})](x_0)=0$,
$\Gamma^k(x_0)=0$, $g^{ij}(x_0)=\delta^j_i$, then
\begin{align}\label{eq13}
E_{{D_V}^*D_V}(x_0)&=\frac{1}{8}\sum_{ijkl}R_{ijkl}\widehat{c}(e_i)\widehat{c}(e_j)c(e_k)c(e_l)-\frac{1}{4}K.\nonumber\\
\end{align}
\indent Similarly, we have
\begin{align}\label{eq13c}
E_{{D_V}^2}(x_0)&=\frac{1}{8}\sum_{ijkl}R_{ijkl}\widehat{c}(e_i)\widehat{c}(e_j)
c(e_k)c(e_l)-\frac{1}{4}K-\frac{1}{4}\sum_{i=1}^n[\widehat{c}(V)\sum_{q=1}^n\widehat{c}(\nabla^L_{e_q}V)c(e_q)c(e_i)]^2\nonumber\\
&+\frac{1}{2}\sum_{j=1}^n[\nabla^{\bigwedge^*T^*M}_{e_j}(\widehat{c}(V)\sum_{q=1}^n\widehat{c}(\nabla^L_{e_q}V)c(e_q))c(e_j)],\nonumber\\
\end{align}
then by (\ref{eq4}), we get Theorem \ref{thm1}.\\
\end{proof}
\indent From \cite{Ac}, we konw that the noncommutative residue of a generalized laplacian $\overline{\Delta}$ is expressed as
\begin{equation}\label{eq14}
(n-2)\Phi_{2}(\overline{\Delta})=(4\pi)^{-\frac{n}{2}}\Gamma(\frac{n}{2})Wres(\overline{\Delta}^{-\frac{n}{2}+1}),
\end{equation}
where $\Phi_{2}(\overline{\Delta})$ denotes the integral over the diagonal part of the second
coefficient of the heat kernel expansion of $\overline{\Delta}$.
Now let $\overline{\Delta}={D_V}^*{D_V}$, since ${D_V}^*{D_V}$ is a generalized laplacian, we can suppose ${D_V}^*{D_V}=\overline{\Delta}-E_{{D_V}^*D_V}$, then we have
\begin{align}\label{eq15}
{\rm Wres}({D_V}^*D_V)^{-\frac{n-2}{2}}
=\frac{(n-2)(4\pi)^{\frac{n}{2}}}{(\frac{n}{2}-1)!}\int_{M}{\rm tr}(\frac{1}{6}K+E_{{D_V}^*D_V})d{\rm Vol_{M} },
\end{align}
\begin{align}\label{eq16}
{\rm Wres}({D_V}^2)^{-\frac{n-2}{2}}
=\frac{(n-2)(4\pi)^{\frac{n}{2}}}{(\frac{n}{2}-1)!}\int_{M}{\rm tr}(\frac{1}{6}K+E_{{D_V}^2})d{\rm Vol_{M} },
\end{align}
where ${\rm Wres}$ denotes the noncommutative residue.\\
Next, we need to compute ${\rm tr}(E_{{D_V}^*D_V})$ and ${\rm tr}(E_{{D_V}^2})$.
Obviously, we have\\
(1)\begin{align}\label{eq17q}
{\rm tr}\bigg(-\frac{1}{4}K\bigg)=-\frac{1}{4}K{\rm tr}[{\rm \texttt{id}}].\nonumber\\
\end{align}
(2)\begin{align}\label{eq17}
&\sum_{ijkl}{\rm tr}[R_{ijkl}\widehat{c}(e_i)\widehat{c}(e_j)
c(e_k)c(e_l)]=0.\nonumber\\
\end{align}
(3)By
 \begin{align}
 \label{eq188}
&\widehat{c}(V)\widehat{c}(\nabla^L_{e_q}V)=-\widehat{c}(\nabla^L_{e_q}V)\widehat{c}(V),~~~ \widehat{c}(V)c(e_q)=-c(e_q)\widehat{c}(V),\nonumber\\
&\widehat{c}(e_q)\widehat{c}(\nabla^L_{e_q}V)+\widehat{c}(\nabla^L_{e_q}V)\widehat{c}(e_q)=2g^{TM}(\nabla^L_{e_q}V,e_q),~~~(\widehat{c}(V))^2=|V|^2=1,\nonumber\\
&\widehat{c}(\nabla^L_{e_q}V)\widehat{c}(\nabla^L_{e_m}V)+\widehat{c}(\nabla^L_{e_m}V)\widehat{c}(\nabla^L_{e_q}V)=2g^{TM}(\nabla^L_{e_q}V,\nabla^L_{e_m}V),\nonumber\\
\end{align}
we also get\\
\begin{align}\label{eq18}
{\rm tr}\sum_{i=1}^n[\widehat{c}(V)\sum_{q=1}^n\widehat{c}(\nabla^L_{e_q}V)c(e_q)c(e_i)]^2&={\rm tr}\sum_{i=1}^n[\widehat{c}(V)\sum_{q=1}^n\widehat{c}(\nabla^L_{e_q}V)c(e_q)c(e_i)\widehat{c}(V)\sum_{m=1}^n\widehat{c}(\nabla^L_{e_m}V)c(e_m)c(e_i)]\nonumber\\
&=-{\rm tr}\sum_{i=1}^n[\sum_{q=1}^n\widehat{c}(\nabla^L_{e_q}V)c(e_q)c(e_i)\sum_{m=1}^n\widehat{c}(\nabla^L_{e_m}V)c(e_m)c(e_i)]\nonumber\\
&=-{\rm tr}\sum_{i=1}^n[\sum_{q=1}^nc(e_q)c(e_i)\widehat{c}(\nabla^L_{e_q}V)\sum_{m=1}^n\widehat{c}(\nabla^L_{e_m}V)c(e_m)c(e_i)]\nonumber\\
&=-2\sum_{i,q,m=1}^ng^{TM}(\nabla^L_{e_q}V,\nabla^L_{e_m}V){\rm tr}[c(e_q)c(e_i)c(e_m)c(e_i)]\nonumber\\
&+{\rm tr}\sum_{i=1}^n[\sum_{q=1}^nc(e_q)c(e_i)\sum_{m=1}^n\widehat{c}(\nabla^L_{e_m}V)\widehat{c}(\nabla^L_{e_q}V)c(e_m)c(e_i)],\nonumber\\
\end{align}
then we have
\begin{align}\label{eq19}
{\rm tr}\sum_{i=1}^n[\sum_{q=1}^n\widehat{c}(\nabla^L_{e_q}V)c(e_q)c(e_i)\sum_{m=1}^n\widehat{c}(\nabla^L_{e_m}V)c(e_m)c(e_i)]=\sum_{i,q,m=1}^ng^{TM}(\nabla^L_{e_q}V,\nabla^L_{e_m}V){\rm tr}[c(e_q)c(e_i)c(e_m)c(e_i)],\nonumber\\
\end{align}
and by $c(e_i)c(e_m)+c(e_m)c(e_i)=-2g^{TM}(e_m,e_i)$, we have
\begin{align}\label{eq20}
&{\rm tr}\sum_{i=1}^n[\sum_{q=1}^n\widehat{c}(\nabla^L_{e_q}V)c(e_q)c(e_i)\sum_{m=1}^n\widehat{c}(\nabla^L_{e_m}V)c(e_m)c(e_i)]\nonumber\\
&=-2\sum_{i,q,m=1}^ng^{TM}(\nabla^L_{e_q}V,\nabla^L_{e_m}V)\delta_{mi}{\rm tr}[c(e_q)c(e_i)]+n\sum_{q,m=1}^ng^{TM}(\nabla^L_{e_q}V,\nabla^L_{e_m}V){\rm tr}[c(e_q)c(e_m)]\nonumber\\
&=-(n-2)\sum_{q=1}^n|\nabla^L_{e_q}V|^2{\rm tr}[{\rm \texttt{id}}],\nonumber\\
\end{align}
\begin{align}
\label{vvv}
{\rm tr}\sum_{i=1}^n[\widehat{c}(V)\sum_{q=1}^n\widehat{c}(\nabla^L_{e_q}V)c(e_q)c(e_i)]^2=(n-2)\sum_{q=1}^n|\nabla^L_{e_q}V|^2{\rm tr}[{\rm \texttt{id}}].\nonumber\\
\end{align}
(4)By $\nabla^{\bigwedge^*T^*M}_{e_j}(\alpha\beta)=(\nabla^{\bigwedge^*T^*M}_{e_j}\alpha)\beta+\alpha(\nabla^{\bigwedge^*T^*M}_{e_j}\beta),$ we have
\begin{align}\label{eq24}
&{\rm tr}\sum_{j=1}^n[\nabla^{\bigwedge^*T^*M}_{e_j}(\widehat{c}(V)\sum_{q=1}^n\widehat{c}(\nabla^L_{e_q}V)c(e_q))c(e_j)]\nonumber\\
&={\rm tr}\sum_{j=1}^n[\nabla^{\bigwedge^*T^*M}_{e_j}(\widehat{c}(V))\sum_{q=1}^n\widehat{c}(\nabla^L_{e_q}V)c(e_q)c(e_j)]+{\rm tr}\sum_{j=1}^n[\widehat{c}(V)\sum_{q=1}^n\nabla^{\bigwedge^*T^*M}_{e_j}(\widehat{c}(\nabla^L_{e_q}V))c(e_q)c(e_j)]\nonumber\\
&+{\rm tr}\sum_{j=1}^n[\widehat{c}(V)\sum_{q=1}^n\widehat{c}(\nabla^L_{e_q}V)\nabla^{\bigwedge^*T^*M}_{e_j}(c(e_q))c(e_j)].\nonumber\\
\end{align}
By
\begin{align}
&\nabla^{\bigwedge^*T^*M}_{e_j}(\widehat{c}(V))=\widehat{c}(\nabla^L_{e_j}V),~~~ \widehat{c}(V)\widehat{c}(\nabla^L_{e_j}(\nabla^L_{e_q}V))+\widehat{c}(\nabla^L_{e_j}(\nabla^L_{e_q}V))\widehat{c}(V)=2g^{TM}(V,\nabla^L_{e_j}(\nabla^L_{e_q}V)),\nonumber\\ &\nabla^{\bigwedge^*T^*M}_{e_j}(\widehat{c}(\nabla^L_{e_q}V)=\widehat{c}(\nabla^L_{e_j}(\nabla^L_{e_q}V)),~~~ g^{TM}(\nabla^L_{e_q}\nabla^L_{e_q}V,V)+g^{TM}(\nabla^L_{e_q}V,\nabla^L_{e_q}V)=e_q(g^{TM}(\nabla^L_{e_q}V,V))=0,\nonumber\\ &\widehat{c}(V)\widehat{c}(\nabla^L_{e_q}V)+\widehat{c}(\nabla^L_{e_q}V)\widehat{c}(V)=0,
~~~\nabla^{\bigwedge^*T^*M}_{e_j}(c(e_q))=c(\nabla^L_{e_j}e_q),\nonumber\\
\end{align}
we get\\
(4-a)\begin{align}\label{eq25}
&{\rm tr}\sum_{j=1}^n[\nabla^{\bigwedge^*T^*M}_{e_j}(\widehat{c}(V))\sum_{q=1}^n\widehat{c}(\nabla^L_{e_q}V)c(e_q)c(e_j)]\nonumber\\
&={\rm tr}\sum_{j=1}^n[\widehat{c}(\nabla^L_{e_j}V)\sum_{q=1}^n\widehat{c}(\nabla^L_{e_q}V)c(e_q)c(e_j)]\nonumber\\
&=2\sum_{j,q=1}^ng^{TM}(\nabla^L_{e_j}V,\nabla^L_{e_q}V){\rm tr}[c(e_q)c(e_j)]-\sum_{j,q=1}^n{\rm tr}[\widehat{c}(\nabla^L_{e_q}V)\widehat{c}(\nabla^L_{e_j}V)c(e_q)c(e_j)],\nonumber\\
\end{align}
similar to (\ref{eq20}), we have
\begin{align}\label{eq26}
{\rm tr}\sum_{j=1}^n[\widehat{c}(\nabla^L_{e_j}V)\sum_{q=1}^n\widehat{c}(\nabla^L_{e_q}V)c(e_q)c(e_j)]=-\sum_{q=1}^n|\nabla^L_{e_q}V|^2{\rm tr}[{\rm \texttt{id}}].\nonumber\\
\end{align}
(4-b)
\begin{align}\label{eq29}
&{\rm tr}\sum_{j=1}^n[\widehat{c}(V)\sum_{q=1}^n\nabla^{\bigwedge^*T^*M}_{e_j}(\widehat{c}(\nabla^L_{e_q}V))c(e_q)c(e_j)]\nonumber\\
&=\sum_{j,q=1}^n{\rm tr}[\widehat{c}(V)\widehat{c}(\nabla^L_{e_j}\nabla^L_{e_q}V)c(e_q)c(e_j)]\nonumber\\
&=2\sum_{j,q=1}^ng^{TM}(\nabla^L_{e_j}\nabla^L_{e_q}V,V){\rm tr}[c(e_q)c(e_j)]-{\rm tr}\sum_{j=1}^n[\sum_{q=1}^n\widehat{c}(\nabla^L_{e_j}\nabla^L_{e_q}V)\widehat{c}(V)c(e_q)c(e_j)],\nonumber\\
\end{align}
then,
\begin{align}\label{c}
\sum_{j,q=1}^n{\rm tr}[\widehat{c}(V)\widehat{c}(\nabla^L_{e_j}\nabla^L_{e_q}V)c(e_q)c(e_j)]&=\sum_{j,q=1}^ng^{TM}(\nabla^L_{e_j}\nabla^L_{e_q}V,V){\rm tr}[c(e_q)c(e_j)]\nonumber\\
&=-\sum_{q=1}^ng^{TM}(\nabla^L_{e_q}\nabla^L_{e_q}V,V){\rm tr}[{\rm \texttt{id}}]\nonumber\\
&=\sum_{q=1}^n|\nabla^L_{e_q}V|^2[{\rm \texttt{id}}].\nonumber\\
\end{align}
(4-c)
\begin{align}\label{eq29}
{\rm tr}\sum_{j=1}^n[\widehat{c}(V)\sum_{q=1}^n\widehat{c}(\nabla^L_{e_q}V)\nabla^{\bigwedge^*T^*M}_{e_j}(c(e_q))c(e_j)]&={\rm tr}\sum_{j=1}^n[\widehat{c}(V)\sum_{q=1}^n\widehat{c}(\nabla^L_{e_q}V)c(\nabla^L_{e_j}e_q)c(e_j)]=0.\nonumber\\
\end{align}
Then, by (\ref{eq26}), (\ref{c}) and (\ref{eq29}), we have
\begin{align}\label{eq219}
{\rm tr}\sum_{j=1}^n[\nabla^{\bigwedge^*T^*M}_{e_j}(\widehat{c}(V)\sum_{q=1}^n\widehat{c}(\nabla^L_{e_q}V)c(e_q))c(e_j)]&=0.\nonumber\\
\end{align}
Therefore, we get
\begin{align}\label{eq31}
&{\rm tr}(E_{{D_V}^*D_V})=-\frac{K}{4}{\rm tr}[{\rm \texttt{id}}],\\
&{\rm tr}(E_{{D_V}^2})=\left(-\frac{K}{4}-\frac{n-2}{4}\sum_{q=1}^n|\nabla^L_{e_q}V|^2\right){\rm tr}[{\rm \texttt{id}}].\nonumber\\
\end{align}
Then by (\ref{eq15}) and (\ref{eq16}), we have the following theorem,
\begin{thm}\label{thm2} If $M$ is a $n$-dimensional compact oriented manifold without boundary, and $n$ is even, then we get the following equalities :
\begin{align}\label{eq32}
&{\rm Wres}({D_V}^*D_V)^{-\frac{n-2}{2}}
=\frac{(n-2)(4\pi)^{\frac{n}{2}}}{(\frac{n}{2}-1)!}\int_{M}2^n\bigg(-\frac{1}{12}K\bigg)d{\rm Vol_{M}},\nonumber\\
&{\rm Wres}({D_V}^2)^{-\frac{n-2}{2}}
=\frac{(n-2)(4\pi)^{\frac{n}{2}}}{(\frac{n}{2}-1)!}\int_{M}2^n
\bigg(-\frac{1}{12}K-\frac{n-2}{4}\sum_{q=1}^n|\nabla^L_{e_q}V|^2\bigg)d{\rm Vol_{M}},
\end{align}
where $K$ is the scalar curvature.
\end{thm}

\section{A Kastler-Kalau-Walze type theorem for $4$-dimensional manifolds with boundary}
\label{Section:3}
 In this section, we prove the Kastler-Kalau-Walze type theorem for $4$-dimensional oriented compact manifolds with boundary. We firstly recall some basic facts and formulas about Boutet de
Monvel's calculus and the definition of the noncommutative residue for manifolds with boundary which will be used in the following. For more details, see in Section 2 in \cite{Wa3}.\\
 \indent Let $M$ be a 4-dimensional compact oriented manifold with boundary $\partial M$.
We assume that the metric $g^{TM}$ on $M$ has the following form near the boundary,
\begin{equation}\label{a1}
g^{TM}=\frac{1}{h(x_{n})}g^{\partial M}+dx _{n}^{2},
\end{equation}
where $g^{\partial M}$ is the metric on $\partial M$ and $h(x_n)\in C^{\infty}([0, 1)):=\{\widehat{h}|_{[0,1)}|\widehat{h}\in C^{\infty}((-\varepsilon,1))\}$ for
some $\varepsilon>0$ and $h(x_n)$ satisfies $h(x_n)>0$, $h(0)=1$ where $x_n$ denotes the normal directional coordinate. Let $U\subset M$ be a collar neighborhood of $\partial M$ which is diffeomorphic with $\partial M\times [0,1)$. By the definition of $h(x_n)\in C^{\infty}([0,1))$
and $h(x_n)>0$, there exists $\widehat{h}\in C^{\infty}((-\varepsilon,1))$ such that $\widehat{h}|_{[0,1)}=h$ and $\widehat{h}>0$ for some
sufficiently small $\varepsilon>0$. Then there exists a metric $g'$ on $\widetilde{M}=M\bigcup_{\partial M}\partial M\times
(-\varepsilon,0]$ which has the form on $U\bigcup_{\partial M}\partial M\times (-\varepsilon,0 ]$
\begin{equation}\label{a2}
g'=\frac{1}{\widehat{h}(x_{n})}g^{\partial M}+dx _{n}^{2} ,
\end{equation}
such that $g'|_{M}=g$. We fix a metric $g'$ on the $\widetilde{M}$ such that $g'|_{M}=g$.

Let the Fourier transformation $F'$  be
\begin{equation*}
F':L^2({\bf R}_t)\rightarrow L^2({\bf R}_v);~F'(u)(v)=\int_\mathbb{R} e^{-ivt}u(t)dt
\end{equation*}
and let
let
\begin{equation*}
r^{+}:C^\infty ({\bf R})\rightarrow C^\infty (\widetilde{{\bf R}^+});~ f\rightarrow f|\widetilde{{\bf R}^+};~
\widetilde{{\bf R}^+}=\{x\geq0;x\in {\bf R}\}.
\end{equation*}
\indent We define $H^+=F'(\Phi(\widetilde{{\bf R}^+}));~ H^-_0=F'(\Phi(\widetilde{{\bf R}^-}))$ which satisfies
$H^+\bot H^-_0$, where $\Phi(\widetilde{{\bf R}^+}) =r^+\Phi({\bf R})$, $\Phi(\widetilde{{\bf R}^-}) =r^-\Phi({\bf R})$ and $\Phi({\bf R})$
denotes the Schwartz space. We have the following
 property: $h\in H^+~$ (resp. $H^-_0$) if and only if $h\in C^\infty({\bf R})$ which has an analytic extension to the lower (resp. upper) complex
half-plane $\{{\rm Im}\xi<0\}$ (resp. $\{{\rm Im}\xi>0\})$ such that for all nonnegative integer $l$,
 \begin{equation*}
\frac{d^{l}h}{d\xi^l}(\xi)\sim\sum^{\infty}_{k=1}\frac{d^l}{d\xi^l}(\frac{c_k}{\xi^k}),
\end{equation*}
as $|\xi|\rightarrow +\infty,{\rm Im}\xi\leq0$ (resp. ${\rm Im}\xi\geq0)$ and where $c_k\in\mathbb{C}$ are some constants.\\
 \indent Let $H'$ be the space of all polynomials and $H^-=H^-_0\bigoplus H';~H=H^+\bigoplus H^-.$ Denote by $\pi^+$ (resp. $\pi^-$) respectively the
 projection on $H^+$ (resp. $H^-$). Let $\widetilde H=\{$rational functions having no poles on the real axis$\}$. Then on $\tilde{H}$,
 \begin{equation}\label{a3}
\pi^+h(\xi_0)=\frac{1}{2\pi i}\lim_{u\rightarrow 0^{-}}\int_{\Gamma^+}\frac{h(\xi)}{\xi_0+iu-\xi}d\xi,
\end{equation}
where $\Gamma^+$ is a Jordan close curve
included ${\rm Im}(\xi)>0$ surrounding all the singularities of $h$ in the upper half-plane and
$\xi_0\in {\bf R}$. In our computations, we only compute $\pi^+h$ for $h$ in $\widetilde{H}$. Similarly, define $\pi'$ on $\tilde{H}$,
\begin{equation}\label{a4}
\pi'h=\frac{1}{2\pi}\int_{\Gamma^+}h(\xi)d\xi.
\end{equation}
So $\pi'(H^-)=0$. For $h\in H\bigcap L^1({\bf R})$, $\pi'h=\frac{1}{2\pi}\int_{{\bf R}}h(v)dv$ and for $h\in H^+\bigcap L^1({\bf R})$, $\pi'h=0$.\\
\indent An operator of order $m\in {\bf Z}$ and type $d$ is a matrix\\
$$\widetilde{A}=\left(\begin{array}{lcr}
  \pi^+P+G  & K  \\
   T  &  \widetilde{S}
\end{array}\right):
\begin{array}{cc}
\   C^{\infty}(M,E_1)\\
 \   \bigoplus\\
 \   C^{\infty}(\partial{M},F_1)
\end{array}
\longrightarrow
\begin{array}{cc}
\   C^{\infty}(M,E_2)\\
\   \bigoplus\\
 \   C^{\infty}(\partial{M},F_2)
\end{array},
$$
where $M$ is a manifold with boundary $\partial M$ and
$E_1,E_2$~ (resp. $F_1,F_2$) are vector bundles over $M~$ (resp. $\partial M
$).~Here,~$P:C^{\infty}_0(\Omega,\overline {E_1})\rightarrow
C^{\infty}(\Omega,\overline {E_2})$ is a classical
pseudodifferential operator of order $m$ on $\Omega$, where
$\Omega$ is a collar neighborhood of $M$ and
$\overline{E_i}|M=E_i~(i=1,2)$. $P$ has an extension:
$~{\cal{E'}}(\Omega,\overline {E_1})\rightarrow
{\cal{D'}}(\Omega,\overline {E_2})$, where
${\cal{E'}}(\Omega,\overline {E_1})~({\cal{D'}}(\Omega,\overline
{E_2}))$ is the dual space of $C^{\infty}(\Omega,\overline
{E_1})~(C^{\infty}_0(\Omega,\overline {E_2}))$. Let
$e^+:C^{\infty}(M,{E_1})\rightarrow{\cal{E'}}(\Omega,\overline
{E_1})$ denote extension by zero from $M$ to $\Omega$ and
$r^+:{\cal{D'}}(\Omega,\overline{E_2})\rightarrow
{\cal{D'}}(\Omega, {E_2})$ denote the restriction from $\Omega$ to
$X$, then define
$$\pi^+P=r^+Pe^+:C^{\infty}(M,{E_1})\rightarrow {\cal{D'}}(\Omega,
{E_2}).$$ In addition, $P$ is supposed to have the
transmission property; this means that, for all $j,k,\alpha$, the
homogeneous component $p_j$ of order $j$ in the asymptotic
expansion of the
symbol $p$ of $P$ in local coordinates near the boundary satisfies:\\
$$\partial^k_{x_n}\partial^\alpha_{\xi'}p_j(x',0,0,+1)=
(-1)^{j-|\alpha|}\partial^k_{x_n}\partial^\alpha_{\xi'}p_j(x',0,0,-1),$$
then $\pi^+P:C^{\infty}(M,{E_1})\rightarrow C^{\infty}(M,{E_2})$
by [12]. Let $G$,$T$ be respectively the singular Green operator
and the trace operator of order $m$ and type $d$. Let $K$ be a
potential operator and $S$ be a classical pseudodifferential
operator of order $m$ along the boundary (For detailed definition,
see [11]). Denote by $B^{m,d}$ the collection of all operators of
order $m$
and type $d$,  and $\mathcal{B}$ is the union over all $m$ and $d$.\\
\indent Recall that $B^{m,d}$ is a Fr\'{e}chet space. The composition
of the above operator matrices yields a continuous map:
$B^{m,d}\times B^{m',d'}\rightarrow B^{m+m',{\rm max}\{
m'+d,d'\}}.$ Write $$\widetilde{A}=\left(\begin{array}{lcr}
 \pi^+P+G  & K \\
 T  &  \widetilde{S}
\end{array}\right)
\in B^{m,d},
 \widetilde{A}'=\left(\begin{array}{lcr}
\pi^+P'+G'  & K'  \\
 T'  &  \widetilde{S}'
\end{array} \right)
\in B^{m',d'}.$$\\
 The composition $\widetilde{A}\widetilde{A}'$ is obtained by
multiplication of the matrices(For more details see [12]). For
example $\pi^+P\circ G'$ and $G\circ G'$ are singular Green
operators of type $d'$ and
$$\pi^+P\circ\pi^+P'=\pi^+(PP')+L(P,P').$$
Here $PP'$ is the usual
composition of pseudodifferential operators and $L(P,P')$ called
leftover term is a singular Green operator of type $m'+d$. For our case, $P,P'$ are classical pseudo differential operators, in other words $\pi^+P\in \mathcal{B}^{\infty}$ and $\pi^+P'\in \mathcal{B}^{\infty}$ .\\
\indent Let $M$ be a $n$-dimensional compact oriented manifold with boundary $\partial M$.
Denote by $\mathcal{B}$ the Boutet de Monvel's algebra. We recall that the main theorem in \cite{FGLS,Wa3}.
\begin{thm}\label{th:32}{\rm\cite{FGLS}}{\bf(Fedosov-Golse-Leichtnam-Schrohe)}
 Let $M$ and $\partial M$ be connected, ${\rm dim}M=n\geq3$, and let $\widetilde{S}$ (resp. $\widetilde{S}'$) be the unit sphere about $\xi$ (resp. $\xi'$) and $\sigma(\xi)$ (resp. $\sigma(\xi')$) be the corresponding canonical
$n-1$ (resp. $(n-2)$) volume form.
 Set $\widetilde{A}=\left(\begin{array}{lcr}\pi^+P+G &   K \\
T &  \widetilde{S}    \end{array}\right)$ $\in \mathcal{B}$ , and denote by $p$, $b$ and $s$ the local symbols of $P,G$ and $\widetilde{S}$ respectively.
 Define:
 \begin{align}
{\rm{\widetilde{Wres}}}(\widetilde{A})&=\int_X\int_{\bf \widetilde{ S}}{\rm{tr}}_E\left[p_{-n}(x,\xi)\right]\sigma(\xi)dx \nonumber\\
&+2\pi\int_ {\partial X}\int_{\bf \widetilde{S}'}\left\{{\rm tr}_E\left[({\rm{tr}}b_{-n})(x',\xi')\right]+{\rm{tr}}
_F\left[s_{1-n}(x',\xi')\right]\right\}\sigma(\xi')dx',
\end{align}
where ${\rm{\widetilde{Wres}}}$ denotes the noncommutative residue of an operator in the Boutet de Monvel's algebra.\\
Then~~ a) ${\rm \widetilde{Wres}}([\widetilde{A},B])=0 $, for any
$\widetilde{A},B\in\mathcal{B}$;~~ b) It is the unique continuous trace on
$\mathcal{B}/\mathcal{B}^{-\infty}$.
\end{thm}
\begin{defn}\label{def1}{\rm\cite{Wa3} }
Lower dimensional volumes of spin manifolds with boundary are defined by
 \begin{equation}\label{a6}
{\rm Vol}^{(p_1,p_2)}_nM:= \widetilde{{\rm Wres}}[\pi^+D^{-p_1}\circ\pi^+D^{-p_2}],
\end{equation}
\end{defn}
 By \cite{Wa3}, we get
\begin{equation}\label{a7}
\widetilde{{\rm Wres}}[\pi^+D^{-p_1}\circ\pi^+D^{-p_2}]=\int_M\int_{|\xi'|=1}{\rm
trace}_{\wedge^*T^*M\bigotimes\mathbb{C}}[\sigma_{-n}(D^{-p_1-p_2})]\sigma(\xi)dx+\int_{\partial M}\Phi,
\end{equation}
and
\begin{eqnarray}\label{a8}
\Phi&=\int_{|\xi'|=1}\int^{+\infty}_{-\infty}\sum^{\infty}_{j, k=0}\sum\frac{(-i)^{|\alpha|+j+k+1}}{\alpha!(j+k+1)!}
\times {\rm trace}_{\wedge^*T^*M\bigotimes\mathbb{C}}[\partial^j_{x_n}\partial^\alpha_{\xi'}\partial^k_{\xi_n}\sigma^+_{r}(D^{-p_1})(x',0,\xi',\xi_n)
\nonumber\\
&\times\partial^\alpha_{x'}\partial^{j+1}_{\xi_n}\partial^k_{x_n}\sigma_{l}(D^{-p_2})(x',0,\xi',\xi_n)]d\xi_n\sigma(\xi')dx',
\end{eqnarray}
 where the sum is taken over $r+l-k-|\alpha|-j-1=-n,~~r\leq -p_1,l\leq -p_2$.

 Since $[\sigma_{-n}(D^{-p_1-p_2})]|_M$ has the same expression as $\sigma_{-n}(D^{-p_1-p_2})$ in the case of manifolds without
boundary, so locally we can compute the first term by \cite{KW}, \cite{Ka}, \cite{Po},  \cite{Wa3}.

For any fixed point $x_0\in\partial M$, we choose the normal coordinates
$U$ of $x_0$ in $\partial M$ (not in $M$) and compute $\Phi(x_0)$ in the coordinates $\widetilde{U}=U\times [0,1)\subset M$ and the
metric $\frac{1}{h(x_n)}g^{\partial M}+dx_n^2.$ The dual metric of $g^{TM}$ on $\widetilde{U}$ is ${h(x_n)}g^{\partial M}+dx_n^2.$  Write
$g^{TM}_{ij}=g^{TM}(\frac{\partial}{\partial x_i},\frac{\partial}{\partial x_j});~ g_{TM}^{ij}=g^{TM}(dx_i,dx_j)$, then

\begin{equation}\label{a9}
[g^{TM}_{ij}]= \left[\begin{array}{lcr}
  \frac{1}{h(x_n)}[g_{ij}^{\partial M}]  & 0  \\
   0  &  1
\end{array}\right];~~~
[g_{TM}^{ij}]= \left[\begin{array}{lcr}
  h(x_n)[g^{ij}_{\partial M}]  & 0  \\
   0  &  1
\end{array}\right],
\end{equation}
and
\begin{equation}\label{a10}
\partial_{x_s}g_{ij}^{\partial M}(x_0)=0, 1\leq i,j\leq n-1; ~~~g_{ij}^{TM}(x_0)=\delta_{ij}.
\end{equation}
\indent From \cite{Wa3}, we can get the following three lemmas,
\begin{lem}{\rm \cite{Wa3}}\label{lem1}
With the metric $g^{TM}$ on $M$ near the boundary
\begin{align}\label{a11}
\partial_{x_j}(|\xi|_{g^{TM}}^2)(x_0)&=\left\{
       \begin{array}{c}
        0,  ~~~~~~~~~~ ~~~~~~~~~~ ~~~~~~~~~~~~~{\rm if }~j<n, \\[2pt]
       h'(0)|\xi'|^{2}_{g^{\partial M}},~~~~~~~~~~~~~~~~~~~~{\rm if }~j=n,
       \end{array}
    \right. \\
\partial_{x_j}(c(\xi))(x_0)&=\left\{
       \begin{array}{c}
      0,  ~~~~~~~~~~ ~~~~~~~~~~ ~~~~~~~~~~~~~{\rm if }~j<n,\\[2pt]
\partial_{x_n}(c(\xi'))(x_{0}), ~~~~~~~~~~~~~~~~~{\rm if }~j=n,
       \end{array}
    \right.
\end{align}
where $\xi=\xi'+\xi_{n}dx_{n}$.
\end{lem}
\begin{lem}{\rm \cite{Wa3}}\label{lem2}With the metric $g^{TM}$ on $M$ near the boundary
\begin{align}\label{a12}
\omega_{s,t}(e_i)(x_0)&=\left\{
       \begin{array}{c}
        \omega_{n,i}(e_i)(x_0)=\frac{1}{2}h'(0),  ~~~~~~~~~~ ~~~~~~~~~~~{\rm if }~s=n,t=i,i<n, \\[2pt]
       \omega_{i,n}(e_i)(x_0)=-\frac{1}{2}h'(0),~~~~~~~~~~~~~~~~~~~{\rm if }~s=i,t=n,i<n,\\[2pt]
    \omega_{s,t}(e_i)(x_0)=0,~~~~~~~~~~~~~~~~~~~~~~~~~~~other~cases,~~~~~~~~~
       \end{array}
    \right.
\end{align}
where $(\omega_{s,t})$ denotes the connection matrix of Levi-Civita connection $\nabla^L$.
\end{lem}
\begin{lem}{\rm \cite{Wa3}}\label{lem3}When $i<n,$ then
\begin{align}
\label{a13}
\Gamma_{st}^k(x_0)&=\left\{
       \begin{array}{c}
        \Gamma^n_{ii}(x_0)=\frac{1}{2}h'(0),~~~~~~~~~~ ~~~~~~~~~~~{\rm if }~s=t=i,k=n, \\[2pt]
        \Gamma^i_{ni}(x_0)=-\frac{1}{2}h'(0),~~~~~~~~~~~~~~~~~~~{\rm if }~s=n,t=i,k=i,\\[2pt]
        \Gamma^i_{in}(x_0)=-\frac{1}{2}h'(0),~~~~~~~~~~~~~~~~~~~{\rm if }~s=i,t=n,k=i,\\[2pt]
       \end{array}
    \right.
\end{align}
in other cases, $\Gamma_{st}^i(x_0)=0$.
\end{lem}
\indent By (\ref{a7}) and (\ref{a8}), we firstly compute
\begin{equation}\label{a14}
\widetilde{{\rm Wres}}[\pi^+{D_V}^{-1}\circ\pi^+({D_V}^*)^{-1}]=\int_M\int_{|\xi'|=1}{\rm
trace}_{\wedge^*T^*M\bigotimes\mathbb{C}}[\sigma_{-4}(({D_V}^*{D_V})^{-1})]\sigma(\xi)dx+\int_{\partial M}\Phi,
\end{equation}
where
\begin{align}\label{a15}
\Phi &=\int_{|\xi'|=1}\int^{+\infty}_{-\infty}\sum^{\infty}_{j, k=0}\sum\frac{(-i)^{|\alpha|+j+k+1}}{\alpha!(j+k+1)!}
\times {\rm trace}_{\wedge^*T^*M\bigotimes\mathbb{C}}[\partial^j_{x_n}\partial^\alpha_{\xi'}\partial^k_{\xi_n}\sigma^+_{r}({D_V}^{-1})(x',0,\xi',\xi_n)
\nonumber\\
&\times\partial^\alpha_{x'}\partial^{j+1}_{\xi_n}\partial^k_{x_n}\sigma_{l}(({D_V}^*)^{-1})(x',0,\xi',\xi_n)]d\xi_n\sigma(\xi')dx',
\end{align}
and the sum is taken over $r+l-k-j-|\alpha|=-3,~~r\leq -1,l\leq-1$.\\

\indent By Theorem \ref{thm2}, we can compute the interior of $\widetilde{{\rm Wres}}[\pi^+{D_V}^{-1}\circ\pi^+({D_V}^*)^{-1}]$, so
\begin{align}\label{a16}
&\int_M\int_{|\xi'|=1}{\rm
trace}_{\wedge^*T^*M}[\sigma_{-4}(({D_V}^*D_V)^{-1})]\sigma(\xi)dx=32\pi^2\int_{M}\bigg(-\frac{4}{3}K\bigg)d{\rm Vol_{M}}.
\end{align}

\indent Now we  need to compute $\int_{\partial M} \Phi$. Since, some operators have the following symbols.
\begin{lem}\label{lem4} The following identities hold:
\begin{align}
\sigma_1({D_V})&=\sigma_1({D_V}^*)=-\widehat{c}(V)c(\xi); \nonumber\\ \sigma_0({D_V})&=\frac{i\widehat{c}(V)}{4}\bigg(\sum_{i,s,t}\omega_{s,t}(e_i)c(e_i)\widehat{c}(e_s)\widehat{c}(e_t)
-\sum_{i,s,t}\omega_{s,t}(e_i)c(e_i)c(e_s)c(e_t)\bigg); \nonumber\\
\sigma_0({D_V}^*)&=\frac{i\widehat{c}(V)}{4}\bigg(\sum_{i,s,t}\omega_{s,t}(e_i)c(e_i)\widehat{c}(e_s)\widehat{c}(e_t)
-\sum_{i,s,t}\omega_{s,t}(e_i)c(e_i)c(e_s)c(e_t)\bigg)-i\sum_{q=1}^nc(e_q)\widehat{c}(\nabla^L_{e_q}V).\nonumber\\
\end{align}
\end{lem}

\indent Write
 \begin{eqnarray}\label{a17}
D_x^{\alpha}&=(-i)^{|\alpha|}\partial_x^{\alpha};
~\sigma(D_V)=p_1+p_0;
~(\sigma(D_V)^{-1})=\sum^{\infty}_{j=1}q_{-j}.
\end{eqnarray}

\indent By the composition formula of pseudodifferential operators, we have
\begin{align}\label{a18}
1=\sigma(D_V\circ {D_V}^{-1})&=\sum_{\alpha}\frac{1}{\alpha!}\partial^{\alpha}_{\xi}[\sigma({D_V})]
{D}_x^{\alpha}[\sigma({D_V}^{-1})]\nonumber\\
&=(p_1+p_0)(q_{-1}+q_{-2}+q_{-3}+\cdots)\nonumber\\
&~~~+\sum_j(\partial_{\xi_j}p_1+\partial_{\xi_j}p_0)(
D_{x_j}q_{-1}+D_{x_j}q_{-2}+D_{x_j}q_{-3}+\cdots)\nonumber\\
&=p_1q_{-1}+(p_1q_{-2}+p_0q_{-1}+\sum_j\partial_{\xi_j}p_1D_{x_j}q_{-1})+\cdots,
\end{align}
so
\begin{equation}\label{a19}
q_{-1}=p_1^{-1};~q_{-2}=-p_1^{-1}[p_0p_1^{-1}+\sum_j\partial_{\xi_j}p_1D_{x_j}(p_1^{-1})].
\end{equation}
\begin{lem}\label{lem6} The following identities hold:
\begin{align}
\sigma_{-1}({D_V}^{-1})&=\sigma_{-1}(({D_V}^*)^{-1})=-\frac{\widehat{c}(V)c(\xi)}{|\xi|^2};\nonumber\\
\sigma_{-2}({D_V}^{-1})&=\frac{c(\xi)\sigma_{0}(D_V)c(\xi)}{|\xi|^4}+\frac{c(\xi)}{|\xi|^6}\sum_jc(dx_j)
\Big[\partial_{x_j}(c(\xi))|\xi|^2-c(\xi)\partial_{x_j}(|\xi|^2)\Big] ;\nonumber\\
\sigma_{-2}(({D_V}^*)^{-1})&=\frac{c(\xi)\sigma_{0}({D_V}^*)c(\xi)}{|\xi|^4}+\frac{c(\xi)}{|\xi|^6}\sum_jc(dx_j)
\Big[\partial_{x_j}(c(\xi))|\xi|^2-c(\xi)\partial_{x_j}(|\xi|^2)\Big].
\end{align}
\end{lem}
\begin{thm}\label{athm1}
Let $M$ be a $4$-dimensional oriented
compact manifold with boundary $\partial M$ and the metric
$g^{TM}$ as in Section \ref{Section:3}, the operators ${D_V}=\sqrt{-1}\widehat{c}(V)(d +\delta)$ and ${D_V}^*=-\sqrt{-1}(d +\delta)\widehat{c}(V)$ be on $\widetilde{M}$ ($\widetilde{M}$ is a collar neighborhood of $M$), then
\begin{align}\label{a20}
&\widetilde{{\rm Wres}}[\pi^+{D_V}^{-1}\circ\pi^+({D_V}^*)^{-1}]\nonumber\\
&=32\pi^2\int_{M}\bigg(-\frac{4}{3}K\bigg)d{\rm Vol_{M}}+\int_{\partial M}\left(-\frac{3ih'(0)}{2}-\frac{27\pi^2h'(0)}{8}-\frac{\pi^2}{4}\right)\pi\Omega_3d{\rm Vol_{M}},
\end{align}
where $K$ is the scalar curvature.
\end{thm}
\begin{proof}
\indent When $n=4$, then ${\rm tr}_{\wedge^*T^*M}[{\rm \texttt{id}}]={\rm dim}(\wedge^*(\mathbb{R}^4))=16$, the sum is taken over $
r+l-k-j-|\alpha|=-3,~~r\leq -1,~~~l\leq-1,$ then we have the following five cases:
~\\
\noindent  {\bf case a)~I)}~$r=-1,~l=-1,~k=j=0,~|\alpha|=1$.\\

\noindent By (\ref{a15}), we get
\begin{equation}\label{a21}
\Phi_1=-\int_{|\xi'|=1}\int^{+\infty}_{-\infty}\sum_{|\alpha|=1}
 {\rm tr}[\partial^\alpha_{\xi'}\pi^+_{\xi_n}\sigma_{-1}({D_V}^{-1})\times
 \partial^\alpha_{x'}\partial_{\xi_n}\sigma_{-1}(({D_V}^*)^{-1})](x_0)d\xi_n\sigma(\xi')dx'.
\end{equation}
By Lemma \ref{lem1}, for $i<n$, we have
\begin{align}\label{a22}
\partial_{x_i}\left(-\frac{\widehat{c}(V)c(\xi)}{|\xi|^2}\right)(x_0)&=-\frac{\partial_{x_i}(\widehat{c}(V))c(\xi)(x_0)}{|\xi|^2}
-\widehat{c}(V)\partial_{x_i}\bigg[\frac{c(\xi)}{|\xi|^2}\bigg](x_0)\nonumber\\
&=-\sum_{l=1}^n\partial_{x_i}(V_l)\widehat{c}(e_l)\frac{c(\xi)}{|\xi|^2}(x_0),\nonumber\\
\end{align}
where $\widehat{c}(V)=\sum_{l=1}^nV_l\widehat{c}(e_l), V_l=g^{TM}(V,e_l).$\\
Then
\begin{align}\label{a23}
&\partial^\alpha_{x'}\partial_{\xi_n}\sigma_{-1}(({D_V}^*)^{-1})\nonumber\\
&=-\frac{\sum_{l=1}^n\partial_{x_i}(V_l)\widehat{c}(e_l)}{(1+\xi_n^2)^2}c(dx_n)+\frac{\xi_n^2\sum_{l=1}^n\partial_{x_i}(V_l)\widehat{c}(e_l)}{(1+\xi_n^2)^2}c(dx_n)+\frac{2\xi_n\sum_{l=1}^n\partial_{x_i}(V_l)\widehat{c}(e_l)}{(1+\xi_n^2)^2}c(\xi').
\end{align}
By $c(\xi)=\sum_{j=1}^n\xi_jc(dx_j), |\xi|^2=\sum_{ij}g^{ij}\xi_i\xi_j$, for $i<n$, we have
\begin{align}\label{a24}
&\partial_{\xi_i}\pi^+_{\xi_n}\left(-\frac{\widehat{c}(V)c(\xi)}{|\xi|^2}\right)(x_0)\nonumber\\
&=\pi^+_{\xi_n}\partial_{\xi_i}\left(-\frac{\widehat{c}(V)\sum_{j=1}^n\xi_jc(dx_j)}{|\xi|^2}\right)(x_0)\nonumber\\
&=\pi^+_{\xi_n}\bigg(\frac{-\widehat{c}(V)c(dx_i)}{|\xi|^2}+\frac{2\sum_{j=1}^n\xi_j\xi_i\widehat{c}(V)c(dx_j)}{|\xi|^4}\bigg)(x_0)\nonumber\\
&=\frac{i}{2(\xi_n-i)}[\widehat{c}(V)c(dx_i)-\sum_{j=1}^{n-1}\xi_j\xi_i\widehat{c}(V)c(dx_j)]-\frac{1}{2(\xi_n-i)^2}\sum_{j=1}^{n-1}\xi_j\xi_i\widehat{c}(V)c(dx_j)-\frac{i}{2(\xi_n-i)^2}\xi_i\widehat{c}(V)c(dx_n).\nonumber\\
\end{align}
Then
\begin{align}\label{a25}
&\sum_{|\alpha|=1}
 {\rm tr}[\partial^\alpha_{\xi'}\pi^+_{\xi_n}\sigma_{-1}({D_V}^{-1})\times
 \partial^\alpha_{x'}\partial_{\xi_n}\sigma_{-1}(({D_V}^*)^{-1})](x_0)\nonumber\\
&=-\frac{i(1-\xi_n^2)}{2(\xi_n-i)^3(\xi_n+i)^2}\sum_{l=1}^n\sum_{i=1}^{n-1}{\rm tr}[\partial_{x_i}(V_l)\widehat{c}(V)c(dx_i)\widehat{c}(e_l)c(dx_n)]\nonumber\\
&+\frac{i\xi_n}{(\xi_n-i)^3(\xi_n+i)^2}\sum_{l=1}^n\sum_{i=1}^{n-1}{\rm tr}[\partial_{x_i}(V_l)\widehat{c}(V)c(dx_i)\widehat{c}(e_l)c(\xi')]\nonumber\\
&+\frac{i(1-\xi_n^2)}{2(\xi_n-i)^3(\xi_n+i)^2}\sum_{l=1}^n\sum_{ij=1}^{n-1}{\rm tr}[\xi_j\xi_i\partial_{x_i}(V_l)\widehat{c}(V)c(dx_j)\widehat{c}(e_l)c(dx_n)]\nonumber\\
&-\frac{i\xi_n}{(\xi_n-i)^3(\xi_n+i)^2}\sum_{l=1}^n\sum_{ij=1}^{n-1}{\rm tr}[\xi_j\xi_i\partial_{x_i}(V_l)\widehat{c}(V)c(dx_j)\widehat{c}(e_l)c(\xi')]\nonumber\\
&+\frac{1-\xi_n^2}{2(\xi_n-i)^4(\xi_n+i)^2}\sum_{l=1}^n\sum_{ij=1}^{n-1}{\rm tr}[\xi_j\xi_i\partial_{x_i}(V_l)\widehat{c}(V)c(dx_j)\widehat{c}(e_l)c(dx_n)]\nonumber\\
&-\frac{\xi_n}{(\xi_n-i)^4(\xi_n+i)^2}\sum_{l=1}^n\sum_{ij=1}^{n-1}{\rm tr}[\xi_j\xi_i\partial_{x_i}(V_l)\widehat{c}(V)c(dx_j)\widehat{c}(e_l)c(\xi')]\nonumber\\
&+\frac{i(1-\xi_n^2)}{2(\xi_n-i)^4(\xi_n+i)^2}\sum_{l=1}^n\sum_{i=1}^{n-1}{\rm tr}[\xi_i\partial_{x_i}(V_l)\widehat{c}(V)c(dx_n)\widehat{c}(e_l)c(dx_n)]\nonumber\\
&-\frac{i\xi_n}{(\xi_n-i)^4(\xi_n+i)^2}\sum_{l=1}^n\sum_{i=1}^{n-1}{\rm tr}[\xi_i\partial_{x_i}(V_l)\widehat{c}(V)c(dx_n)\widehat{c}(e_l)c(\xi')].\nonumber\\
\end{align}
\noindent By $\widehat{c}(V)\widehat{c}(e_l)+\widehat{c}(e_l)\widehat{c}(V)=2g^{TM}(V,e_l)=2V_l$ and ${\rm tr}{ab}={\rm tr }{ba}$,\\
\begin{align}\label{a226}
&\sum_{l=1}^n\sum_{i=1}^{n-1}{\rm tr}[\partial_{x_i}(V_l)\widehat{c}(V)c(dx_i)\widehat{c}(e_l)c(dx_n)]\nonumber\\
&=\sum_{l=1}^n\sum_{i=1}^{n-1}\partial_{x_i}(V_l){\rm tr}[\widehat{c}(V)c(dx_i)\widehat{c}(e_l)c(dx_n)]\nonumber\\
&=-\sum_{l=1}^n\sum_{i=1}^{n-1}\partial_{x_i}(V_l){\rm tr}[c(dx_i)\widehat{c}(V)\widehat{c}(e_l)c(dx_n)]\nonumber\\
&=-\sum_{l=1}^n\sum_{i=1}^{n-1}2V_l\partial_{x_i}(V_l){\rm tr}[c(dx_i)c(dx_n)]+\sum_{l=1}^n\sum_{i=1}^{n-1}\partial_{x_i}(V_l){\rm tr}[c(dx_i)\widehat{c}(e_l)\widehat{c}(V)c(dx_n)]\nonumber\\
&=-\sum_{l=1}^n\sum_{i=1}^{n-1}2V_l\partial_{x_i}(V_l){\rm tr}[c(dx_i)c(dx_n)]-\sum_{l=1}^n\sum_{i=1}^{n-1}\partial_{x_i}(V_l){\rm tr}[\widehat{c}(V)c(dx_i)\widehat{c}(e_l)c(dx_n)],\nonumber\\
\end{align}
then,
\begin{align}
\label{111}
\sum_{l=1}^n\sum_{i=1}^{n-1}\partial_{x_i}(V_l){\rm tr}[\widehat{c}(V)c(dx_i)\widehat{c}(e_l)c(dx_n)]=-\sum_{l=1}^n\sum_{i=1}^{n-1}V_l\partial_{x_i}(V_l){\rm tr}[c(dx_i)c(dx_n)]=0.\nonumber\\
\end{align}
\begin{align}\label{555}
&\sum_{l=1}^n\sum_{ij=1}^{n-1}{\rm tr}[\xi_j\xi_i\partial_{x_i}(V_l)\widehat{c}(V)c(dx_j)\widehat{c}(e_l)c(\xi')]\nonumber\\
&=\sum_{l=1}^n\sum_{ij=1}^{n-1}\xi_j\xi_i\partial_{x_i}(V_l){\rm tr}[\widehat{c}(V)c(dx_j)\widehat{c}(e_l)c(\xi')]\nonumber\\
&=-\sum_{l=1}^n\sum_{ij=1}^{n-1}\xi_j\xi_i\xi_k\partial_{x_i}(V_l){\rm tr}[c(dx_j)\widehat{c}(V)\widehat{c}(e_l)c(dx_k)]\nonumber\\
&=-\sum_{l=1}^n\sum_{ij=1}^{n-1}\xi_j\xi_i\xi_k2V_l\partial_{x_i}(V_l){\rm tr}[c(dx_j)c(dx_k)]+\sum_{l=1}^n\sum_{ij=1}^{n-1}\xi_j\xi_i\xi_k\partial_{x_i}(V_l){\rm tr}[c(dx_j)\widehat{c}(e_l)\widehat{c}(V)c(dx_k)],\nonumber\\
\end{align}
then, by $\partial_{x_i}(V_l)V_l=\frac{1}{2}\partial_{x_i}((V_l)^2)=0,$
\begin{align}
\label{191}
\sum_{l=1}^n\sum_{ij=1}^{n-1}\xi_j\xi_i\partial_{x_i}(V_l){\rm tr}[\widehat{c}(V)c(dx_j)\widehat{c}(e_l)c(\xi')]=-\sum_{l=1}^n\sum_{ij=1}^{n-1}\xi_j\xi_i\xi_kV_l\partial_{x_i}(V_l){\rm tr}[c(dx_j)c(dx_k)]=0.\nonumber\\
\end{align}
Similarly, we have the following equalities:
\begin{align}\label{a26}
&\sum_{l=1}^n\sum_{i=1}^{n-1}{\rm tr}[\partial_{x_i}(V_l)\widehat{c}(V)c(dx_i)\widehat{c}(e_l)c(\xi')]=0;~~~\sum_{l=1}^n\sum_{i=1}^{n-1}{\rm tr}[\xi_i\partial_{x_i}(V_l)\widehat{c}(V)c(dx_n)\widehat{c}(e_l)c(\xi')]=0;\nonumber\\
&\sum_{l=1}^n\sum_{ij=1}^{n-1}{\rm tr}[\xi_j\xi_i\partial_{x_i}(V_l)\widehat{c}(V)c(dx_j)\widehat{c}(e_l)c(dx_n)]=0;~~~\sum_{l=1}^n\sum_{ij=1}^{n-1}{\rm tr}[\xi_j\xi_i\partial_{x_i}(V_l)\widehat{c}(V)c(dx_j)\widehat{c}(e_l)c(dx_n)]=0.\nonumber\\
\end{align}
Therefore,
\begin{align}\label{a27}
&\Phi_1=-\int_{|\xi'|=1}\int^{+\infty}_{-\infty}\sum_{|\alpha|=1}
 {\rm tr}[\partial^\alpha_{\xi'}\pi^+_{\xi_n}\sigma_{-1}({D_V}^{-1})\times
 \partial^\alpha_{x'}\partial_{\xi_n}\sigma_{-1}(({D_V}^*)^{-1})](x_0)d\xi_n\sigma(\xi')dx'\nonumber\\
 &=0.\nonumber\\
\end{align}
 \noindent  {\bf case a)~II)}~$r=-1,~l=-1,~k=|\alpha|=0,~j=1$.\\

\noindent By (\ref{a15}), we get
\begin{equation}\label{a28}
\Phi_2=-\frac{1}{2}\int_{|\xi'|=1}\int^{+\infty}_{-\infty} {\rm
trace} [\partial_{x_n}\pi^+_{\xi_n}\sigma_{-1}({D_V}^{-1})\times
\partial_{\xi_n}^2\sigma_{-1}(({D_V}^*)^{-1})](x_0)d\xi_n\sigma(\xi')dx'.
\end{equation}
\noindent By Lemma \ref{lem6}, we have\\
\begin{eqnarray}\label{a29}\partial^2_{\xi_n}\sigma_{-1}(({D_V}^*)^{-1})(x_0)=\widehat{c}(V)\left(\frac{6\xi_nc(dx_n)+2c(\xi')}
{|\xi|^4}-\frac{8\xi_n^2c(\xi)}{|\xi|^6}\right);
\end{eqnarray}
\begin{eqnarray}\label{a30}
\partial_{x_n}\sigma_{-1}({D_V}^{-1})(x_0)=-\frac{\partial_{x_n}(\widehat{c}(V))c(\xi)}{|\xi|^2}(x_0)-\frac{\widehat{c}(V)\partial_{x_n}(c(\xi'))(x_0)}{|\xi|^2}+\frac{\widehat{c}(V)c(\xi)|\xi'|^2h'(0)}{|\xi|^4}(x_0).
\end{eqnarray}
By (\ref{a3}), (\ref{a4}) and the Cauchy integral formula we have
\begin{align}\label{a31}
\pi^+_{\xi_n}\left[-\frac{\partial_{x_n}(\widehat{c}(V))c(\xi)}{|\xi|^2}\right](x_0)|_{|\xi'|=1}&=-\partial_{x_n}(\widehat{c}(V))\pi^+_{\xi_n}\left[\frac{(i\xi_n+2)c(\xi')+ic(dx_n)}{1+\xi_n^2}\right]\nonumber\\
&=-\partial_{x_n}(\widehat{c}(V))\frac{1}{2\pi i}{\rm lim}_{u\rightarrow
0^-}\int_{\Gamma^+}\frac{\frac{c(\xi')+\eta_nc(dx_n)}{(\eta_n+i)(\xi_n+iu-\eta_n)}}
{(\eta_n-i)}d\eta_n\nonumber\\
&=i\partial_{x_n}(\widehat{c}(V))\frac{c(\xi')+ic(dx_n)}{2(\xi_n-i)}.
\end{align}
Similarly, we have,
\begin{eqnarray}\label{a32}
\pi^+_{\xi_n}\left[\frac{\widehat{c}(V)\partial_{x_n}(c(\xi'))}{|\xi|^2}\right](x_0)|_{|\xi'|=1}=\frac{\widehat{c}(V)\partial_{x_n}(c(\xi'))(x_0)}{2(\xi_n-i)};
\end{eqnarray}
\begin{eqnarray}\label{a33}
\pi^+_{\xi_n}\left[\frac{\widehat{c}(V)c(\xi)|\xi'|^2h'(0)}{|\xi|^4}\right](x_0)|_{|\xi'|=1}=-ih'(0)\widehat{c}(V)\left[\frac{(i\xi_n+2)c(\xi')+ic(dx_n)}{4(\xi_n-i)^2}\right].
\end{eqnarray}
By (\ref{a30}), then\\
\begin{align}\label{a34}
&\pi^+_{\xi_n}\partial_{x_n}(\sigma_{-1}({D_V}^{-1}))|_{|\xi'|=1}\nonumber\\
&=\frac{i\partial_{x_n}(\widehat{c}(V))(c(\xi')+ic(dx_n))-\widehat{c}(V)\partial_{x_n}(c(\xi'))}{2(\xi_n-i)}-ih'(0)\widehat{c}(V)\left[\frac{(i\xi_n+2)c(\xi')+ic(dx_n)}{4(\xi_n-i)^2}\right].
\end{align}
\noindent Similar to (\ref{a226})-(\ref{191}), we have the following equalities:
\begin{align}\label{a35}
&{\rm tr}[\partial_{x_n}(\widehat{c}(V))c(\xi')\widehat{c}(V)c(dx_n)]=0;~~{\rm tr}[\partial_{x_n}(\widehat{c}(V))c(\xi')\widehat{c}(V)c(\xi')]=0;\nonumber\\
&{\rm tr}[\partial_{x_n}(\widehat{c}(V))c(dx_n)\widehat{c}(V)c(dx_n)]=0;~~{\rm tr}[\partial_{x_n}(\widehat{c}(V))c(dx_n)\widehat{c}(V)c(\xi')]=0;\nonumber\\
&{\rm tr}[\widehat{c}(V)\partial_{x_n}(c(\xi'))\widehat{c}(V)c(\xi')]=8h'(0)\sum_{k=1}^{n-1}\xi_k^2;~~~{\rm tr}[\widehat{c}(V)c(\xi')\widehat{c}(V)c(\xi)]=16\sum_{k=1}^{n-1}\xi_k^2;\nonumber\\
&{\rm tr}[\widehat{c}(V)c(dx_n)\widehat{c}(V)c(\xi')]=0;~~~{\rm tr}[\widehat{c}(V)\partial_{x_n}(c(\xi'))\widehat{c}(V)c(dx_n)]=0.\nonumber\\
\end{align}
By (\ref{a29}), (\ref{a34}) and (\ref{a35}), we have
\begin{align}\label{a36}
&{\rm
trace} [\partial_{x_n}\pi^+_{\xi_n}\sigma_{-1}({D_V}^{-1})\times
\partial_{\xi_n}^2\sigma_{-1}(({D_V}^*)^{-1})](x_0)\nonumber\\
&=\left[\frac{(24\xi_n^2+8\xi_n-8-16i)h'(0)}{(\xi_n-i)^4(\xi_n+i)^2}+\frac{(64\xi_n^2i-32\xi_n^3)h'(0)}{(\xi_n-i)^5(\xi_n+i)^3}\right]\sum_{k=1}^{n-1}\xi_k^2+\frac{(24\xi_n-8\xi_n^3)h'(0)}{(\xi_n-i)^5(\xi_n+i)^3}.\nonumber\\
\end{align}
Considering $\int_{S^3}\xi_i\xi_j=\frac{\pi^2}{2}\delta^{ij}$ see (\cite{Ka}),
then\\
\begin{align}\label{a37}
\Phi_2&=-\frac{1}{2}\int_{|\xi'|=1}\int^{+\infty}_{-\infty} {\rm
trace} [\partial_{x_n}\pi^+_{\xi_n}\sigma_{-1}({D_V}^{-1})\times
\partial_{\xi_n}^2\sigma_{-1}(({D_V}^*)^{-1})](x_0)d\xi_n\sigma(\xi')dx'\nonumber\\
&=-\frac{1}{2}\int_{|\xi'|=1}\int^{+\infty}_{-\infty}\left[\frac{(24\xi_n^2+8\xi_n-8-16i)h'(0)}{(\xi_n-i)^4(\xi_n+i)^2}+\frac{(64\xi_n^2i-32\xi_n^3)h'(0)}{(\xi_n-i)^5(\xi_n+i)^3}\right]\sum_{k=1}^{n-1}\xi_k^2d\xi_n\sigma(\xi')dx'\nonumber\\
&-\frac{1}{2}\int_{|\xi'|=1}\int^{+\infty}_{-\infty}\frac{(24\xi_n-8\xi_n^3)h'(0)}{(\xi_n-i)^5(\xi_n+i)^3}d\xi_n\sigma(\xi')dx'\nonumber\\
&=-4h'(0)\Omega_3\frac{2\pi i}{3!}\frac{\pi^2}{2}[\frac{(3\xi_n^2+\xi_n-1-2i)}{(\xi_n+i)^2}]^{(3)}|_{\xi_n=i}dx'-16h'(0)\frac{\pi^2}{2}\Omega_3\frac{2\pi i}{4!}h'(0)[\frac{(2\xi_n^2i-\xi_n^3)}{(\xi_n+i)^3}]^{(4)}|_{\xi_n=i}dx'\nonumber\\
&-4h'(0)\Omega_3\frac{2\pi i}{4!}[\frac{3\xi_n-\xi_n^3}{(\xi_n+i)^3}]^{(4)}|_{\xi_n=i}dx'\nonumber\\
&=\bigg(\frac{(1-6i)h'(0)}{4}+\frac{(1-6i)\pi^2h'(0)}{8}\bigg)\pi\Omega_3dx',\nonumber\\
\end{align}
where ${\rm \Omega_{3}}$ is the canonical volume of $S^{3}.$\\
\noindent  {\bf case a)~III)}~$r=-1,~l=-1,~j=|\alpha|=0,~k=1$.\\

\noindent By (\ref{a15}), we get
\begin{equation}\label{a38}
\Phi_3=-\frac{1}{2}\int_{|\xi'|=1}\int^{+\infty}_{-\infty}
{\rm trace} [\partial_{\xi_n}\pi^+_{\xi_n}\sigma_{-1}({D_V}^{-1})\times
\partial_{\xi_n}\partial_{x_n}\sigma_{-1}(({D_V}^*)^{-1})](x_0)d\xi_n\sigma(\xi')dx'.
\end{equation}
\noindent By Lemma \ref{lem6}, we have\\
\begin{align}\label{a39}
&\partial_{\xi_n}\partial_{x_n}\sigma_{-1}(({D_V}^*)^{-1})(x_0)|_{|\xi'|=1}\nonumber\\
&=-\frac{\partial_{x_n}(\widehat{c}(V))c(dx_n)}{|\xi|^2}+\frac{2\partial_{x_n}(\widehat{c}(V))(\xi_nc(\xi')+\xi_n^2c(dx_n))}{|\xi|^4}+\frac{2\widehat{c}(V)\xi_n\partial_{x_n}(c(\xi'))(x_0)}{|\xi|^4}\nonumber\\
&+h'(0)\left(\frac{\widehat{c}(V)c(dx_n)}{|\xi|^4}-4\xi_n\frac{\widehat{c}(V)(c(\xi')+\xi_nc(dx_n))}{|\xi|^6}\right),\nonumber\\
\end{align}
\begin{eqnarray}\label{a40}
\partial_{\xi_n}\pi^+_{\xi_n}\sigma_{-1}({D_V}^{-1})(x_0)|_{|\xi'|=1}=-\frac{i\widehat{c}(V)c(\xi')-\widehat{c}(V)c(dx_n)}{2(\xi_n-i)^2},
\end{eqnarray}
By (\ref{a35}), (\ref{a39})  and (\ref{a40}), we have\\
\begin{align}\label{a41}
&{\rm trace} [\partial_{\xi_n}\pi^+_{\xi_n}\sigma_{-1}({D_V}^{-1})\times
\partial_{\xi_n}\partial_{x_n}\sigma_{-1}(({D_V}^*)^{-1})](x_0)\nonumber\\
&=\frac{8h'(0)}{(\xi_n-i)^4(\xi_n+i)^2}+4h'(0)\frac{3\xi_ni-4\xi_n^2-\xi_n^3i}{(\xi_n-i)^5(\xi_n+i)^3}\sum_{l=1}^{n-1}\xi_k^2.\nonumber\\
\end{align}
Then,
\begin{align}\label{a42}
\Phi_3&=-\frac{1}{2}\int_{|\xi'|=1}\int^{+\infty}_{-\infty}
{\rm trace} [\partial_{\xi_n}\pi^+_{\xi_n}\sigma_{-1}({D_V}^{-1})\times
\partial_{\xi_n}\partial_{x_n}\sigma_{-1}(({D_V}^*)^{-1})](x_0)d\xi_n\sigma(\xi')dx'\nonumber\\
&=-\int_{|\xi'|=1}\int^{+\infty}_{-\infty}\frac{2h'(0)(3\xi_ni-4\xi_n^2-\xi_n^3i)}
{(\xi_n-i)^5(\xi_n+i)^3}\sum_{k=1}^{n-1}\xi_k^2d\xi_n\sigma(\xi')dx'-\int_{|\xi'|=1}\int^{+\infty}_{-\infty}\frac{4h'(0)}
{(\xi_n-i)^4(\xi_n+i)^2}d\xi_n\sigma(\xi')dx'\nonumber\\
&=-2h'(0)\Omega_3\frac{\pi^2}{2}\frac{2\pi i}{4!}[\frac{3\xi_ni-4\xi_n^2-\xi_n^3i}{(\xi_n+i)^3}]^{(4)}|_{\xi_n=i}dx'-4h'(0)\Omega_3\frac{2\pi i}{3!}[\frac{1}{(\xi_n+i)^2}]^{(3)}|_{\xi_n=i}dx'\nonumber\\
&=\left(-\frac{2\pi^2-\pi^2h'(0)}{8}-h'(0)\right)\pi\Omega_3dx'.
\end{align}

\noindent  {\bf case b)}~$r=-2,~l=-1,~k=j=|\alpha|=0$.\\

\noindent By (\ref{a15}), we get
\begin{align}\label{a43}
\Phi_4&=-i\int_{|\xi'|=1}\int^{+\infty}_{-\infty}{\rm trace} [\pi^+_{\xi_n}\sigma_{-2}({D_V}^{-1})\times
\partial_{\xi_n}\sigma_{-1}(({D_V}^*)^{-1})](x_0)d\xi_n\sigma(\xi')dx'.
\end{align}
 By Lemma \ref{lem6} we have\\
\begin{align}\label{a44}
\sigma_{-2}({D_V}^{-1})(x_0)=\frac{c(\xi)\sigma_{0}({D_V})(x_0)c(\xi)}{|\xi|^4}+\frac{c(\xi)}{|\xi|^6}c(dx_n)
[\partial_{x_n}(c(\xi'))(x_0)|\xi|^2-c(\xi)h'(0)|\xi|^2_{\partial
M}],
\end{align}
where
\begin{align}\label{a45}
\sigma_{0}({D_V})(x_0)&=\frac{i\widehat{c}(V)}{4}\sum_{s,t,i}\omega_{s,t}(e_i)
(x_{0})c(e_i)\widehat{c}(e_s)\widehat{c}(e_t)
-\frac{i\widehat{c}(V)}{4}\sum_{s,t,i}\omega_{s,t}(e_i)
(x_{0})c(e_i)c(e_s)c(e_t).\nonumber\\
\end{align}
We denote
\begin{align}\label{a46}
A_0^{1}(x_0)&=i\widehat{c}(V)\frac{1}{4}\sum_{s,t,i}\omega_{s,t}(e_i)
(x_{0})c(e_i)\widehat{c}(e_s)\widehat{c}(e_t)=i\widehat{c}(V)a_0^{1}(x_0);\nonumber\\
A_0^{2}(x_0)&=-i\widehat{c}(V)\frac{1}{4}\sum_{s,t,i}\omega_{s,t}(e_i)
(x_{0})c(e_i)c(e_s)c(e_t)=i\widehat{c}(V)a_0^{2}(x_0),
\end{align}
where $a_0^{2}=c_0c(dx_n)$ and $c_0=-\frac{3}{4}h'(0)$.\\
Then
\begin{align}\label{a47}
\pi^+_{\xi_n}\sigma_{-2}({D_V}^{-1}(x_0))|_{|\xi'|=1}&=\pi^+_{\xi_n}\Big[\frac{c(\xi)A_0^{1}(x_0)c(\xi)}{(1+\xi_n^2)^2}\Big]
\nonumber\\
&+\pi^+_{\xi_n}\Big[\frac{c(\xi)A_0^{2}(x_0)c(\xi)+c(\xi)c(dx_n)\partial_{x_n}(c(\xi'))(x_0)}{(1+\xi_n^2)^2}-h'(0)\frac{c(\xi)c(dx_n)c(\xi)}{(1+\xi_n^{2})^3}\Big].
\end{align}
By computations, we have
\begin{align}\label{a48}
\pi^+_{\xi_n}\Big[\frac{c(\xi)A_0^{1}(x_0)c(\xi)}{(1+\xi_n^2)^2}\Big]&=\pi^+_{\xi_n}\Big[\frac{c(\xi')A_0^{1}(x_0)c(\xi')}{(1+\xi_n^2)^2}\Big]
+\pi^+_{\xi_n}\Big[ \frac{\xi_nc(\xi')A_0^{1}(x_0)c(dx_{n})}{(1+\xi_n^2)^2}\Big]\nonumber\\
&+\pi^+_{\xi_n}\Big[\frac{\xi_nc(dx_{n})A_0^{1}(x_0)c(\xi')}{(1+\xi_n^2)^2}\Big]
+\pi^+_{\xi_n}\Big[\frac{\xi_n^{2}c(dx_{n})A_0^{1}(x_0)c(dx_{n})}{(1+\xi_n^2)^2}\Big]\nonumber\\
&=-\frac{c(\xi')A_0^{1}(x_0)c(\xi')(2+i\xi_{n})}{4(\xi_{n}-i)^{2}}
+\frac{ic(\xi')A_0^{1}(x_0)c(dx_{n})}{4(\xi_{n}-i)^{2}}\nonumber\\
&+\frac{ic(dx_{n})A_0^{1}(x_0)c(\xi')}{4(\xi_{n}-i)^{2}}
+\frac{-i\xi_{n}c(dx_{n})A_0^{1}(x_0)c(dx_{n})}{4(\xi_{n}-i)^{2}}.
\end{align}
Since
\begin{align}\label{a49}
c(dx_n)a_0^{1}(x_0)
&=-\frac{1}{4}h'(0)\sum^{n-1}_{i=1}c(e_i)
\widehat{c}(e_i)c(e_n)\widehat{c}(e_n),
\end{align}
then by the relation of the Clifford action and ${\rm tr}{ab}={\rm tr }{ba}$,  we have the following equalities:\\
\begin{align}\label{a50}
&{\rm tr}[c(e_i)
\widehat{c}(e_i)c(e_n)
\widehat{c}(e_n)]=0~~(i<n);~~
{\rm tr}[a_0^{1}c(dx_n)]=0;~~~{\rm tr}[\widehat{c}(\xi')\widehat{c}(dx_n)]=0.
\end{align}
Since
\begin{align}\label{a51}
\partial_{\xi_n}\sigma_{-1}(({D_V}^*)^{-1})=-\widehat{c}(V)\left[\frac{c(dx_n)}{1+\xi_n^2}-\frac{2\xi_nc(\xi')+2\xi_n^2c(dx_n)}{(1+\xi_n^2)^2}\right].
\end{align}
By (\ref{a48}) and (\ref{a51}), we have
\begin{align}\label{a52}
&{\rm tr }[\pi^+_{\xi_n}\Big[\frac{c(\xi)A_0^{1}(x_0)c(\xi)}{(1+\xi_n^2)^2}\Big]
\times\partial_{\xi_n}\sigma_{-1}(({D_V}^*)^{-1})(x_0)]|_{|\xi'|=1}\nonumber\\
&=-\frac{1}{2(1+\xi_n^2)^2}{\rm tr }[c(\xi')a_0^{1}(x_0)]
-\frac{i}{2(1+\xi_n^2)^2}{\rm tr }[c(dx_n)a_0^{1}(x_0)]\nonumber\\
&=-\frac{1}{2(1+\xi_n^2)^2}{\rm tr }[c(\xi')a_0^{1}(x_0)].
\end{align}
We note that $i<n,~\int_{|\xi'|=1}\{\xi_{i_{1}}\xi_{i_{2}}\cdots\xi_{i_{2d+1}}\}\sigma(\xi')=0$,
so ${\rm tr }[c(\xi')a_0^{1}(x_0)]$ has no contribution for computing {\rm case~b)}.

By computations, we have
\begin{eqnarray}\label{a53}
\pi^+_{\xi_n}\Big[\frac{c(\xi)a_0^{2}(x_0)c(\xi)+c(\xi)c(dx_n)\partial_{x_n}(c(\xi'))(x_0)}{(1+\xi_n^2)^2}\Big]-h'(0)\pi^+_{\xi_n}\Big[\frac{c(\xi)c(dx_n)c(\xi)}{(1+\xi_n)^3}\Big]:= N_1-N_2,
\end{eqnarray}
where
\begin{align}\label{a54}
N_1&=\frac{-1}{4(\xi_n-i)^2}[i(2+i\xi_n)c(\xi')\widehat{c}(V)a_0^{2}(x_0)c(\xi')-\xi_nc(dx_n)\widehat{c}(V)a_0^{2}(x_0)c(dx_n)\nonumber\\
&+(2+i\xi_n)c(\xi')c(dx_n)\partial_{x_n}(c(\xi'))-c(dx_n)\widehat{c}(V)a_0^{2}(x_0)c(\xi')
-c(\xi')\widehat{c}(V)a_0^{2}(x_0)c(dx_n)-i\partial_{x_n}(c(\xi'))]
\end{align}
and
\begin{align}\label{a55}
N_2&=\frac{h'(0)}{2}\left[\frac{c(dx_n)}{4i(\xi_n-i)}+\frac{c(dx_n)-ic(\xi')}{8(\xi_n-i)^2}
+\frac{3\xi_n-7i}{8(\xi_n-i)^3}[ic(\xi')-c(dx_n)]\right],
\end{align}
then by the relation of the Clifford action and ${\rm tr}{ab}={\rm tr }{ba}$,  we have the following equalities:\\
\begin{align}\label{a56}
&{\rm tr}[c(\xi')
\widehat{c}(V)
c(\xi')]=0;~~
{\rm tr}[c(dx_n)
\widehat{c}(V)
c(\xi')]=0;\nonumber\\
&{\rm tr}[c(\xi')
\widehat{c}(V)
c(dx_n)]=0;~~
{\rm tr}[c(dx_n)
\widehat{c}(V)
c(dx_n)]=0.\nonumber\\
\end{align}
By (\ref{a51}), (\ref{a55}) and (\ref{a56}), we have\\
\begin{align}\label{a57}
&{\rm tr }[N_2\times\partial_{\xi_n}\sigma_{-1}(({D_V}^*)^{-1})]|_{|\xi'|=1}\nonumber\\
&=\frac{ih'(0)\xi_n(\xi_n-3i)}{4(\xi_n-i)^5(\xi_n+i)^2}{\rm tr}[c(\xi')
\widehat{c}(V)
c(\xi')]+\frac{h'(0)(4\xi_ni-\xi_n^3i-3\xi_n^2)}{4(\xi_n-i)^5(\xi_n+i)^2}{\rm tr}[c(dx_n)
\widehat{c}(V)
c(\xi')]\nonumber\\
&+\frac{ih'(0)(\xi_n^2-1)(\xi_n-3i)}{8(\xi_n-i)^5(\xi_n+i)^2}{\rm tr}[c(\xi')
\widehat{c}(V)
c(dx_n)]+\frac{ih'(0)\xi_n(\xi_n-3i)}{4(\xi_n-i)^5(\xi_n+i)^2}{\rm tr}[c(dx_n)
\widehat{c}(V)
c(\xi')]\nonumber\\
&=0.\nonumber\\
\end{align}
By (\ref{a51}), (\ref{a54}) and (\ref{a56}), we have
\begin{eqnarray}\label{a58}{\rm tr }[N_1\times\partial_{\xi_n}\sigma_{-1}(({D_V}^*)^{-1})]|_{|\xi'|=1}=
\frac{3h'(0)\xi_n}{(\xi_n-i)^3(\xi_n+i)}+3h'(0)\frac{\xi_n^3-2i\xi_n^2-5\xi_n+2i}{(\xi_n-i)^4(\xi_n+i)^2}\sum_{k=1}^{n-1}\xi_k^2.\nonumber\\
\end{eqnarray}
By (\ref{a57}) and (\ref{a58}), we have
\begin{align}\label{a59}
\Phi_4&=-i\int_{|\xi'|=1}\int^{+\infty}_{-\infty}{\rm trace} [\pi^+_{\xi_n}\sigma_{-2}({D_V}^{-1})\times
\partial_{\xi_n}\sigma_{-1}(({D_V}^*)^{-1})](x_0)d\xi_n\sigma(\xi')dx'\nonumber\\
&=-i\int_{|\xi'|=1}\int^{+\infty}_{-\infty}\frac{3h'(0)\xi_n}{(\xi_n-i)^3(\xi_n+i)}d\xi_n\sigma(\xi')dx'\nonumber\\
&-i\int_{|\xi'|=1}\int^{+\infty}_{-\infty}3h'(0)\frac{\xi_n^3-2i\xi_n^2-5\xi_n+2i}{(\xi_n-i)^4(\xi_n+i)^2}\sum_{k=1}^{n-1}\xi_k^2d\xi_n\sigma(\xi')dx'\nonumber\\
&=-3h'(0)i\Omega_3\frac{2\pi i}{2!}[\frac{\xi_n}{(\xi_n+i)}]^{(2)}|_{\xi_n=i}dx'-3h'(0)i\frac{\pi^2}{2}\Omega_3\frac{2\pi i}{3!}[\frac{\xi_n^3-2i\xi_n^2-5\xi_n+2i}{(\xi_n+i)^2}]^{(3)}|_{\xi_n=i}dx'\nonumber\\
&=\frac{(6-9\pi^2)h'(0)}{8}\pi\Omega_3dx'.\nonumber\\
\end{align}
\noindent {\bf  case c)}~$r=-1,~l=-2,~k=j=|\alpha|=0$.\\
By (\ref{a15}), we get
\begin{align}\label{a60}
\Phi_5=-i\int_{|\xi'|=1}\int^{+\infty}_{-\infty}{\rm trace} [\pi^+_{\xi_n}\sigma_{-1}({D_V}^{-1})\times
\partial_{\xi_n}\sigma_{-2}(({D_V}^*)^{-1})](x_0)d\xi_n\sigma(\xi')dx'.
\end{align}
By (\ref{a3}) and (\ref{a4}), Lemma \ref{lem6}, we have
\begin{align}\label{a61}
\pi^+_{\xi_n}\sigma_{-1}({D_V}^{-1})|_{|\xi'|=1}=\frac{i\widehat{c}(V)c(\xi')-\widehat{c}(V)c(dx_n)}{2(\xi_n-i)}.
\end{align}
Since
\begin{equation}\label{a62}
\sigma_{-2}(({D_V}^*)^{-1})(x_0)=\frac{c(\xi)\sigma_{0}({D_V}^*)(x_0)c(\xi)}{|\xi|^4}+\frac{c(\xi)}{|\xi|^6}c(dx_n)
\bigg[\partial_{x_n}(c(\xi'))(x_0)|\xi|^2-c(\xi)h'(0)|\xi|^2_{\partial_
M}\bigg],
\end{equation}
where
\begin{align}\label{a63}
&\sigma_{0}({D_V}^*)(x_0)\nonumber\\
&=i\widehat{c}(V)\frac{1}{4}\sum_{s,t,i}\omega_{s,t}(e_i)(x_{0})c(e_i)\widehat{c}(e_s)\widehat{c}(e_t)
-i\widehat{c}(V)\frac{1}{4}\sum_{s,t,i}\omega_{s,t}(e_i)(x_{0})c(e_i)c(e_s)c(e_t)-i\sum_{q=1}^nc(e_q)\widehat{c}(\nabla^L_{e_q}V)(x_0)\nonumber\\
&=i\widehat{c}(V)a_0^1(x_0)+i\widehat{c}(V)a_0^2(x_0)-i\sum_{q=1}^nc(e_q)\widehat{c}(\nabla^L_{e_q}V)(x_0),\nonumber\\
\end{align}
then,
\begin{align}\label{a64}
&\partial_{\xi_n}\sigma_{-2}(({D_V}^*)^{-1})(x_0)|_{|\xi'|=1}\nonumber\\
&=
\partial_{\xi_n}\bigg\{\frac{c(\xi)[A_0^{1}(x_0)+A_0^{2}(x_0)
-i\sum_{q=1}^nc(e_q)\widehat{c}(\nabla^L_{e_q}V)(x_0)]c(\xi)}{|\xi|^4}+\frac{c(\xi)}{|\xi|^6}c(dx_n)[\partial_{x_n}(c(\xi'))(x_0)|\xi|^2-c(\xi)h'(0)]\bigg\}\nonumber\\
&=\partial_{\xi_n}\bigg\{\frac{c(\xi)A_0^{1}(x_0)c(\xi)}{|\xi|^4}+\frac{c(\xi)}{|\xi|^6}c(dx_n)[\partial_{x_n}(c(\xi'))(x_0)|\xi|^2-c(\xi)h'(0)]\bigg\}+\partial_{\xi_n}\frac{c(\xi)A_0^{2}(x_0)c(\xi)}{|\xi|^4}\nonumber\\
&-\partial_{\xi_n}\frac{c(\xi)[i\sum_{q=1}^nc(e_q)\widehat{c}(\nabla^L_{e_q}V)](x_0)c(\xi)}{|\xi|^4}\nonumber\\
&=M_1+M_2-M_3,\nonumber\\
\end{align}
where
\begin{align*} &M_1=\partial_{\xi_n}\frac{c(\xi)A_0^{1}(x_0)c(\xi)}{|\xi|^4},\\
&M_2=\partial_{\xi_n}\left\{\frac{c(\xi)A_0^{2}(x_0)c(\xi)}{|\xi|^4}+\frac{c(\xi)}{|\xi|^6}c(dx_n)[\partial_{x_n}(c(\xi'))(x_0)|\xi|^2-c(\xi)h'(0)]\right\},\\ &M_3=\partial_{\xi_n}\frac{c(\xi)[i\sum_{q=1}^nc(e_q)\widehat{c}(\nabla^L_{e_q}V)](x_0)c(\xi)}{|\xi|^4}.\\
\end{align*}
By computations, we have
\begin{align}\label{a65}
M_1&=\partial_{\xi_n}\frac{c(\xi)A_0^{1}(x_0)c(\xi)}{|\xi|^4}\nonumber\\
&=i\frac{c(dx_n)\widehat{c}(V)a_0^{1}(x_0)c(\xi)}{|\xi|^4}
+i\frac{c(\xi)\widehat{c}(V)a_0^{1}(x_0)c(dx_n)}{|\xi|^4}
-i\frac{4\xi_n c(\xi)\widehat{c}(V)a_0^{1}(x_0)c(\xi)}{|\xi|^6};
\end{align}
\begin{align}\label{a66}
M_2&=\frac{1}{(1+\xi_n^2)^3}\bigg[(2\xi_n-2\xi_n^3)c(dx_n)A_0^{2}c(dx_n)
+(1-3\xi_n^2)c(dx_n)A_0^{2}c(\xi')\nonumber\\
&+(1-3\xi_n^2)c(\xi')A_0^{2}c(dx_n)
-4\xi_nc(\xi')A_0^{2}c(\xi')
+(3\xi_n^2-1){\partial}_{x_n}(c(\xi'))\nonumber\\
&-4\xi_nc(\xi')c(dx_n){\partial}_{x_n}(c(\xi'))
+2h'(0)c(\xi')+2h'(0)\xi_nc(dx_n)\bigg]\nonumber\\
&+6\xi_nh'(0)\frac{c(\xi)c(dx_n)c(\xi)}{(1+\xi^2_n)^4};\nonumber\\
\end{align}
\begin{align}\label{a67}
M_3&=\partial_{\xi_n}\frac{c(\xi)[i\sum_{q=1}^nc(e_q)\widehat{c}(\nabla^L_{e_q}V)](x_0)c(\xi)}{|\xi|^4}\nonumber\\
&=\frac{c(dx_n)[i\sum_{q=1}^nc(e_q)\widehat{c}(\nabla^L_{e_q}V)](x_0)c(\xi)}{|\xi|^4}
+\frac{c(\xi)[i\sum_{q=1}^nc(e_q)\widehat{c}(\nabla^L_{e_q}V)](x_0)c(dx_n)}{|\xi|^4}\nonumber\\
&-\frac{4\xi_n c(\xi)[i\sum_{q=1}^nc(e_q)\widehat{c}(\nabla^L_{e_q}V)](x_0)c(\xi)}{|\xi|^4}.
\end{align}
By (\ref{a61}) and (\ref{a65}), we have
\begin{align}\label{a68}
&{\rm tr}[\pi^+_{\xi_n}\sigma_{-1}({D_V}^{-1})\times
\partial_{\xi_n}\frac{c(\xi)A_0^{1}c(\xi)}
{|\xi|^4}](x_0)|_{|\xi'|=1}\nonumber\\
&=\frac{1}{(\xi-i)(\xi+i)^3}{\rm tr}[c(\xi')a_0^{1}(x_0)]
-\frac{i}{(\xi-i)(\xi+i)^3}{\rm tr}[c(dx_n)a_0^{1}(x_0)]\nonumber\\
&=\frac{1}{(\xi-i)(\xi+i)^3}{\rm tr}[c(\xi')a_0^{1}(x_0)].\nonumber\\
\end{align}
We note that $i<n,~\int_{|\xi'|=1}\{\xi_{i_{1}}\xi_{i_{2}}\cdots\xi_{i_{2d+1}}\}\sigma(\xi')=0$,
so ${\rm tr }[c(\xi')a_0^{1}(x_0)]$ has no contribution for computing {\rm case~c)}.
By (\ref{a61}) and (\ref{a66}), we have
\begin{eqnarray}\label{a69}
{\rm tr}[\pi^+_{\xi_n}\sigma_{-1}({D_V}^{-1})\times
M_2](x_0)|_{|\xi'|=1}
=\frac{12h'(0)(\xi_n^3-\xi_n)}{(\xi-i)^4(\xi+i)^3}+\frac{6h'(0)(1-3\xi_n^2-4\xi_n)}{(\xi-i)^4(\xi+i)^3}\sum_{k=1}^{n-1}\xi_k^2.\nonumber\\
\end{eqnarray}
By (\ref{a61}) and (\ref{a67}), we have
\begin{align}\label{a70}
{\rm tr}[\pi^+_{\xi_n}\sigma_{-1}({D_V}^{-1})\times
M_3](x_0)|_{|\xi'|=1}&=\frac{1}{2(\xi-i)^3(\xi+i)^2}{\rm tr}[\widehat{c}(V)c(\xi')c(dx_n)\sum_{q=1}^nc(e_q)\widehat{c}(\nabla^L_{e_q}V)c(\xi)]\nonumber\\
&-\frac{i}{2(\xi-i)^3(\xi+i)^2}{\rm tr}[\widehat{c}(V)\sum_{q=1}^nc(e_q)\widehat{c}(\nabla^L_{e_q}V)c(\xi)]\nonumber\\
&+\frac{(1-4\xi_n)}{2(\xi-i)^3(\xi+i)^2}{\rm tr}[\widehat{c}(V)c(\xi')c(\xi)\sum_{q=1}^nc(e_q)\widehat{c}(\nabla^L_{e_q}V)c(\xi)]\nonumber\\
&+\frac{(1-4\xi_n)i}{2(\xi-i)^3(\xi+i)^2}{\rm tr}[\widehat{c}(V)c(dx_n)c(\xi)\sum_{q=1}^nc(e_q)\widehat{c}(\nabla^L_{e_q}V)c(\xi)].\nonumber\\
\end{align}
And  by the relation of the Clifford action and ${\rm tr}{ab}={\rm tr }{ba}$,  we have the following equalities:
\begin{align}\label{a71}
{\rm tr }[\widehat{c}(V)c(\xi')c(dx_n)\sum_{q=1}^nc(e_q)\widehat{c}(\nabla^L_{e_q}V)c(\xi')]=0,~~~{\rm tr }[\widehat{c}(V)\sum_{q=1}^nc(e_q)\widehat{c}(\nabla^L_{e_q}V)c(dx_n)]=0.\nonumber\\
\end{align}
Then, we get
\begin{align}\label{a71}
\Phi_5&=-i\int_{|\xi'|=1}\int^{+\infty}_{-\infty}{\rm trace} [\pi^+_{\xi_n}\sigma_{-1}({D_V}^{-1})\times
\partial_{\xi_n}\sigma_{-2}(({D_V}^*)^{-1})](x_0)d\xi_n\sigma(\xi')dx'\nonumber\\
&=-i\int_{|\xi'|=1}\int^{+\infty}_{-\infty}\frac{12h'(0)(\xi_n^3-\xi_n)}{(\xi-i)^4(\xi+i)^3}d\xi_ndx'-i\int_{|\xi'|=1}\int^{+\infty}_{-\infty}\frac{6h'(0)(1-3\xi_n^2-4\xi_n)}{(\xi-i)^4(\xi+i)^3}\sum_{k=1}^{n-1}\xi_k^2d\xi_ndx'\nonumber\\
&=-12h'(0)i\Omega_3\frac{2\pi i}{3!}[\frac{\xi_n^3-\xi_n}{(\xi_n+i)}]^{(3)}|_{\xi_n=i}-6h'(0)i\Omega_3\frac{2\pi i}{3!}\frac{\pi^2}{2}[\frac{1-3\xi_n^2-4\xi_n}{(\xi_n+i)}]^{(3)}|_{\xi_n=i}dx'\nonumber\\
&=\frac{3(i-2)\pi^2h'(0)}{4}\pi\Omega_3dx'.\nonumber\\
\end{align}
So,
\begin{align}\label{a73}
\Phi=\sum_{i=1}^5\Phi_i=\left(-\frac{3ih'(0)}{2}-\frac{27\pi^2h'(0)}{8}-\frac{\pi^2}{4}\right)\pi\Omega_3dx'.
\end{align}
Then, by (\ref{a14})-(\ref{a16}), we obtain Theorem \ref{athm1}.
\end{proof}
Next, we also prove the Kastler-Kalau-Walze type theorem for $4$-dimensional manifolds with boundary associated to ${D_V}^2$.
By (\ref{a7}) and (\ref{a8}), we will compute
\begin{equation}\label{a74}
\widetilde{{\rm Wres}}[\pi^+{D_V}^{-1}\circ\pi^+{D_V}^{-1}]=\int_M\int_{|\xi'|=1}{\rm
trace}_{\wedge^*T^*M\bigotimes\mathbb{C}}[\sigma_{-4}({D_V}^{-2})]\sigma(\xi)dx+\int_{\partial M}\overline{\Phi},
\end{equation}
where
\begin{align}\label{a75}
\overline{\Phi} &=\int_{|\xi'|=1}\int^{+\infty}_{-\infty}\sum^{\infty}_{j, k=0}\sum\frac{(-i)^{|\alpha|+j+k+1}}{\alpha!(j+k+1)!}
\times {\rm trace}_{\wedge^*T^*M\bigotimes\mathbb{C}}[\partial^j_{x_n}\partial^\alpha_{\xi'}\partial^k_{\xi_n}\sigma^+_{r}({D_V}^{-1})(x',0,\xi',\xi_n)
\nonumber\\
&\times\partial^\alpha_{x'}\partial^{j+1}_{\xi_n}\partial^k_{x_n}\sigma_{l}({D_V}^{-1})(x',0,\xi',\xi_n)]d\xi_n\sigma(\xi')dx',
\end{align}
and the sum is taken over $r+l-k-j-|\alpha|=-3,~~r\leq -1,l\leq-1$.\\

By Theorem \ref{thm2}, we compute the interior of $\widetilde{{\rm Wres}}[\pi^+{D_V}^{-1}\circ\pi^+{D_V}^{-1}]$, then
\begin{align}\label{a76}
&\int_M\int_{|\xi'|=1}{\rm
trace}_{\wedge^*T^*M\bigotimes\mathbb{C}}[\sigma_{-4}({D_V}^{-2})]\sigma(\xi)dx=32\pi^2\int_{M}
\bigg(
-\frac{4}{3}K-8\sum_{q=1}^4|\nabla^L_{e_q}V|^2\bigg)d{\rm Vol_{M}}.\nonumber\\
\end{align}
\begin{thm}\label{athm2}
Let $M$ be a $4$-dimensional oriented
compact manifold with boundary $\partial M$ and the metric
$g^{TM}$ as in Section \ref{Section:3}, the operator ${D_V}=\sqrt{-1}\widehat{c}(V)(d +\delta)$ be on $\widetilde{M}$ ($\widetilde{M}$ is a collar neighborhood of $M$), then
\begin{align}\label{a77}
&\widetilde{{\rm Wres}}[\pi^+{D_V}^{-1}\circ\pi^+{D_V}^{-1}]\nonumber\\
&=32\pi^2\int_{M}\bigg(
-\frac{4}{3}K-8\sum_{q=1}^4|\nabla^L_{e_q}V|^2\bigg)d{\rm Vol_{M}}+\int_{\partial M}\left(-\frac{3ih'(0)}{2}-\frac{27\pi^2h'(0)}{8}-\frac{\pi^2}{4}\right)\pi\Omega_3d{\rm Vol_{M}},\nonumber\\
\end{align}
where $K$ is the scalar curvature.
\end{thm}
\begin{proof}
When $n=4$, by Lemma \ref{lem6}, $\sigma_{-1}({D_V}^{-1})=\sigma_{-1}(({D_V}^*)^{-1})$, then we have the following five cases:\\
\noindent  {\bf case a)~I)}~$r=-1,~l=-1,~k=j=0,~|\alpha|=1$.\\
\begin{align}\label{a78}
\overline{\Phi}_1&=-\int_{|\xi'|=1}\int^{+\infty}_{-\infty}\sum_{|\alpha|=1}
 {\rm tr}[\partial^\alpha_{\xi'}\pi^+_{\xi_n}\sigma_{-1}({D_V}^{-1})\times
 \partial^\alpha_{x'}\partial_{\xi_n}
 \sigma_{-1}({D_V}^{-1})](x_0)d\xi_n\sigma(\xi')dx'\nonumber\\
 &=0.\nonumber\\
\end{align}
\noindent  {\bf case a)~II)}~$r=-1,~l=-1,~k=|\alpha|=0,~j=1$.\\
\begin{align}\label{a79}
\overline{\Phi}_2&=-\frac{1}{2}\int_{|\xi'|=1}\int^{+\infty}_{-\infty} {\rm
trace} [\partial_{x_n}\pi^+_{\xi_n}\sigma_{-1}({D_V}^{-1})\times
\partial_{\xi_n}^2\sigma_{-1}({D_V}^{-1})](x_0)d\xi_n\sigma(\xi')dx'\nonumber\\
&=\bigg(\frac{(1-6i)h'(0)}{4}+\frac{(1-6i)\pi^2h'(0)}{8}\bigg)\pi\Omega_3dx'.\nonumber\\
\end{align}
\noindent  {\bf case a)~III)}~$r=-1,~l=-1,~j=|\alpha|=0,~k=1$.\\
\begin{align}\label{a80}
\overline{\Phi}_3&=-\frac{1}{2}\int_{|\xi'|=1}\int^{+\infty}_{-\infty}
{\rm trace} [\partial_{\xi_n}\pi^+_{\xi_n}\sigma_{-1}({D_V}^{-1})\times
\partial_{\xi_n}\partial_{x_n}\sigma_{-1}({D_V}^{-1})](x_0)d\xi_n\sigma(\xi')dx'\nonumber\\
&=\left(-\frac{2\pi^2-\pi^2h'(0)}{8}-h'(0)\right)\pi\Omega_3dx'.
\end{align}
\noindent  {\bf case b)}~$r=-2,~l=-1,~k=j=|\alpha|=0$.\\
\begin{align}\label{a81}
\overline{\Phi}_4&=-i\int_{|\xi'|=1}\int^{+\infty}_{-\infty}{\rm trace} [\pi^+_{\xi_n}\sigma_{-2}({D_V}^{-1})\times
\partial_{\xi_n}\sigma_{-1}({D_V}^{-1})](x_0)d\xi_n\sigma(\xi')dx'\nonumber\\
&=\frac{(6-9\pi^2)h'(0)}{8}\pi\Omega_3dx'.\nonumber\\
\end{align}
\noindent {\bf  case c)}~$r=-1,~l=-2,~k=j=|\alpha|=0$.\\
By (\ref{a75}), we get
\begin{equation}\label{a82}
\overline{\Phi}_5=-i\int_{|\xi'|=1}\int^{+\infty}_{-\infty}{\rm trace} [\pi^+_{\xi_n}\sigma_{-1}({D_V}^{-1})\times
\partial_{\xi_n}\sigma_{-2}({D_V}^{-1})](x_0)d\xi_n\sigma(\xi')dx'.
\end{equation}
By (\ref{a3}) and (\ref{a4}), Lemma \ref{lem6}, we have
\begin{align}\label{a83}
\pi^+_{\xi_n}\sigma_{-1}({D_V}^{-1})|_{|\xi'|=1}=\frac{i\widehat{c}(V)c(\xi')-\widehat{c}(V)c(dx_n)}{2(\xi_n-i)}.
\end{align}

Since
\begin{equation}\label{a84}
\sigma_{-2}(D_V^{-1})(x_0)=\frac{c(\xi)\sigma_{0}(D_V)(x_0)c(\xi)}{|\xi|^4}+\frac{c(\xi)}{|\xi|^6}c(dx_n)
\bigg[\partial_{x_n}(c(\xi'))(x_0)|\xi|^2-c(\xi)h'(0)|\xi|^2_{\partial_
M}\bigg],
\end{equation}
where
\begin{align}\label{a85}
&\sigma_{0}(D_V)(x_0)\nonumber\\
&=i\widehat{c}(V)\frac{1}{4}\sum_{s,t,i}\omega_{s,t}(e_i)(x_{0})c(e_i)\widehat{c}(e_s)\widehat{c}(e_t)
-i\widehat{c}(V)\frac{1}{4}\sum_{s,t,i}\omega_{s,t}(e_i)(x_{0})c(e_i)c(e_s)c(e_t)\nonumber\\
&=i\widehat{c}(V)a_0^1(x_0)+i\widehat{c}(V)a_0^2(x_0),\nonumber\\
\end{align}
then
\begin{align}\label{a86}
&\partial_{\xi_n}\sigma_{-2}(D_V^{-1})(x_0)|_{|\xi'|=1}\nonumber\\
&=
\partial_{\xi_n}\bigg\{\frac{c(\xi)[A_0^{1}(x_0)+A_0^{2}(x_0)]c(\xi)}{|\xi|^4}+\frac{c(\xi)}{|\xi|^6}c(dx_n)[\partial_{x_n}(c(\xi'))(x_0)|\xi|^2-c(\xi)h'(0)]\bigg\}\nonumber\\
&=\partial_{\xi_n}\bigg\{\frac{c(\xi)A_0^{1}(x_0)c(\xi)}{|\xi|^4}+\frac{c(\xi)}{|\xi|^6}c(dx_n)[\partial_{x_n}(c(\xi'))(x_0)|\xi|^2-c(\xi)h'(0)]\bigg\}+\partial_{\xi_n}\frac{c(\xi)A_0^{2}(x_0)c(\xi)}{|\xi|^4}\nonumber\\
&=M_1+M_2.\nonumber\\
\end{align}
By (\ref{a68}) and (\ref{a69}), we have
\begin{align}\label{a87}
\overline{\Phi}_5&=-i\int_{|\xi'|=1}\int^{+\infty}_{-\infty}{\rm trace} [\pi^+_{\xi_n}\sigma_{-1}({D_V}^{-1})\times
\partial_{\xi_n}\sigma_{-2}(({D_V}^*)^{-1})](x_0)d\xi_n\sigma(\xi')dx'\nonumber\\
&=-i\int_{|\xi'|=1}\int^{+\infty}_{-\infty}\frac{12h'(0)(\xi_n^3-\xi_n)}{(\xi-i)^4(\xi+i)^3}d\xi_ndx'-i\int_{|\xi'|=1}\int^{+\infty}_{-\infty}\frac{6h'(0)(1-3\xi_n^2-4\xi_n)}{(\xi-i)^4(\xi+i)^3}\sum_{k=1}^{n-1}\xi_k^2d\xi_ndx'\nonumber\\
&=-12h'(0)i\Omega_3\frac{2\pi i}{3!}[\frac{\xi_n^3-\xi_n}{(\xi_n+i)}]^{(3)}|_{\xi_n=i}-6h'(0)i\Omega_3\frac{2\pi i}{3!}\frac{\pi^2}{2}[\frac{1-3\xi_n^2-4\xi_n}{(\xi_n+i)}]^{(3)}|_{\xi_n=i}dx'\nonumber\\
&=\frac{3(i-2)\pi^2h'(0)}{4}\pi\Omega_3dx'.\nonumber\\
\end{align}

Therefore, we get
\begin{align}\label{a88}
\overline{\Phi}=\sum_{i=1}^5\overline{\Phi}_i=\left(-\frac{3ih'(0)}{2}-\frac{27\pi^2h'(0)}{8}-\frac{\pi^2}{4}\right)\pi\Omega_3dx'.
\end{align}
By (\ref{a74})-(\ref{a76}), we obtain Theorem \ref{athm2}.\\
\end{proof}
\section{The operator $\sqrt{-1}\widehat{c}(V)(d+\delta)$ for $3$-dimensioanl Spin Manifolds with Boundary}
\label{Section:4}
For an odd-dimensional manifolds with boundary, as in section 5,6 and 7 in \cite{Wa1}, we have the formula\\
\begin{align}\label{b1}
\widetilde{{\rm Wres}}[(\pi^+D_V^{-1})^2]=\int_{\partial M}\Psi.
\end{align}
When $n=3$, then in (\ref{a8}), $r-k-|\alpha|+l-j-1=-3$, $r,l \leq -1,$ so we get $r=l=-1,k=|\alpha|=j=0,$ then
\begin{eqnarray}\label{b2}
\Psi&=\int_{|\xi'|=1}\int^{+\infty}_{-\infty}{\rm trace}_{S(TM)}[\sigma^+_{-1}(D_V^{-1})(x',0,\xi',\xi_n)
\nonumber\\
&\times\partial_{\xi_n}\sigma_{-1}(D_V^{-1})(x',0,\xi',\xi_n)]d\xi_3\sigma(\xi')dx'.
\end{eqnarray}
By Lemma \ref{lem4}, we have
\begin{align}\label{b3}
&\sigma_{-1}^+({D_V}^{-1})|_{|\xi'=1|}=-\frac{\widehat{c}(V)[c(\xi')+ic(dx_n)]}{2i(\xi_n-i)};\nonumber\\
&\partial_{\xi_n}\sigma_{-1}^+({D_V}^{-1})|_{|\xi'=1|}=-\frac{\widehat{c}(V)c(dx_n)}{1+\xi_n^2}+\frac{2\xi_n\widehat{c}(V)c(\xi)}{(1+\xi_n^2)^2}.
\end{align}
For  $n=3$, we take the coordinates in Section 2. Locally $S(TM)|_{\widetilde{U}}\cong \widetilde{U}\times\bigwedge^{even}_\mathbf{C}(2).$ Let $\{\widetilde{f}_1, \widetilde{f}_2\}$ be an orthonormal basis of $\bigwedge^{even}_\mathbf{C}(2)$ and we will compute the trace under this basis.\\
By ${\rm tr}[c(\xi')c(dx_3)]=0;$ ${\rm tr}[c(dx_3)^2]=-8;$ ${\rm tr}[c(\xi')^2]|_{|\xi'=1|}=-8,$ we get\\
\begin{align}\label{b4}
&{\rm tr}[\widehat{c}(V)c(\xi')\widehat{c}(V)c(dx_3)]=0;~~~{\rm tr}[\widehat{c}(V)c(dx_3)\widehat{c}(V)c(dx_3)]=8;\nonumber\\
&{\rm tr}[\widehat{c}(V)c(\xi')\widehat{c}(V)c(\xi)]=8;~~~{\rm tr}[\widehat{c}(V)c(dx_3)\widehat{c}(V)c(\xi)]=8\xi_n.\nonumber\\
\end{align}
Then, by (\ref{b3}) and (\ref{b4}), we have
\begin{align}\label{b5}
{\rm trace}_{S(TM)}[\sigma^+_{-1}(D_V^{-1})
&\times\partial_{\xi_n}\sigma_{-1}(D_V^{-1})](x_0)|_{|\xi'=1|}=-\frac{4}{(\xi_n+i)^2(\xi_n-i)}.\nonumber\\
\end{align}
By (\ref{b2}) and (\ref{b5}) and the Cauchy integral formula, we get\\
\begin{align}\label{b6}
\Phi=2i\pi\Omega_2vol_{\partial M}=4i\pi^2vol_{\partial M},\nonumber\\
\end{align}
where $vol_{\partial M}$ denotes the canonical volume form of $\partial M$.\\
Therefore, we get the following theorem
\begin{thm}\label{athm3}
Let $M$ be a $3$-dimensional
oriented compact manifold with boundary $\partial M$ and the metric
$g^M$ as in Section \ref{Section:3}, the operator ${D_V}=\sqrt{-1}\widehat{c}(V)(d +\delta)$ be on $\widetilde{M}$ ($\widetilde{M}$ is a collar neighborhood of $M$), then
\begin{align}\label{b7}
\widetilde{{\rm Wres}}[(\pi^+D_V^{-1})^2]=4i\pi^2vol_{\partial M},\nonumber\\
\end{align}
where $vol_{\partial M}$ denotes the canonical volume form of $\partial M$.\\
\end{thm}

\section*{Acknowledgements}
This work was supported by NSFC. 11771070 .
 The authors thank the referee for his (or her) careful reading and helpful comments.

\section*{}

\clearpage
\section*{Declarations}
\noindent a. Funding Information: This work was supported by National Natural Science Foundation of China: No.11771070. Author Tong Wu and Yong Wang have received research support from National Natural Science Foundation of China: No.11771070.\\

\noindent b. Competing Interests: The authors have no relevant financial or non-financial interests to disclose.\\

\noindent c. Author Contributions: All authors contributed to the study conception and design. Material preparation, data collection and analysis were performed by Tong Wu and Yong Wang. The first draft of the manuscript was written by Tong Wu and all authors commented on previous versions of the manuscript. All authors read and approved the final manuscript.
\end{document}